\documentclass[11pt,reqno]{amsart}
\usepackage{fullpage}
\usepackage{xcolor}
\usepackage[T1]{fontenc}                          
\usepackage{amsfonts}
\usepackage[utf8]{inputenc}                       
\usepackage{comment}                              
\usepackage{mparhack}                             
\usepackage{amsmath,amssymb,amsthm,mathrsfs,eucal,mathtools}
\usepackage{enumerate}
\usepackage{enumitem}
\usepackage{booktabs}         
\usepackage{float}
\usepackage{graphicx,subfig}   
\usepackage{subcaption}
\usepackage{cancel}
\usepackage{wrapfig}                            
\usepackage[bookmarks=true, colorlinks=true]{hyperref}
\hypersetup{urlcolor=blue,citecolor=red,linkcolor=black}
\usepackage{cleveref}
\usepackage{soul}
\usepackage{cite}

\usepackage{appendix}

\usepackage{bm}                                   
\usepackage{dsfont}

\newtheorem{defn}{Definition}[section]
\newtheorem{thm}[defn]{Theorem}
\newtheorem{prop}[defn]{Proposition}
\newtheorem{lem}[defn]{Lemma}
\newtheorem{cor}[defn]{Corollary}

\newtheorem{rem}[defn]{Remark}
\numberwithin{equation}{section}

\newcommand{\me}{\mathcal{E}}
\newcommand{\mP}{{\mathcal{P}}}

\newcommand{\mf}{\mathcal{F}}

\newcommand{\ee}{\varepsilon}
\newcommand{\Dee}{\mathcal{D}_\ee}

\newcommand{\R}{\mathbb{R}}
\newcommand{\Rd}{{\mathbb{R}^{d}}}
\newcommand{\Rn}{{\mathbb{R}^{n}}}
\newcommand{\T}{{\mathbb{T}}}
\newcommand{\Td}{{\mathbb{T}^{d}}}

\newcommand{\diff}{\mathop{}\!\mathrm{d}}
\newcommand{\dx}{\diff x}
\newcommand{\dy}{\diff y}
\newcommand{\dz}{\diff z}
\newcommand{\dt}{\diff t}
\newcommand{\ds}{\diff s}

\newcommand{\p}{\partial}
\newcommand{\omegat}{\widetilde{\omega}_{\eps}}
\newcommand{\f}{\frac}
\newcommand{\eps}{\varepsilon}

\newcommand{\dive}{\mbox{div}}

\DeclarePairedDelimiter{\norm}{\lVert}{\rVert}

\makeatletter
\@namedef{subjclassname@2020}{\textup{2020} Mathematics Subject Classification}
\makeatother

\def\XXint#1#2#3{{\setbox0=\hbox{$#1{#2#3}{\int}$}
  \vcenter{\hbox{$#2#3$}}\kern-.5\wd0}}

\definecolor{cadmiumgreen}{rgb}{0.0, 0.42, 0.24}

\title[Deterministic particle approximation for a fourth order equation]{A deterministic particle approximation for a fourth-order equation}
\author[C. Elbar, A. Fernández-Jiménez]{Charles Elbar, Alejandro Fernández-Jiménez}

\address{Université Claude Bernard Lyon 1, ICJ UMR5208, CNRS, Ecole Centrale de Lyon, INSA Lyon, Université Jean Monnet, 69622 Villeurbanne, France.}

\email{elbar@math.univ-lyon1.fr}

\address{Université Claude Bernard Lyon 1, ICJ UMR5208, CNRS, Ecole Centrale de Lyon, INSA Lyon, Université Jean Monnet, 69622 Villeurbanne, France. Vrije Universiteit Amsterdam,
    Department of Mathematics,
    Faculteit der Exacte Wetenschappen,
    De Boelelaan 1081a,
    1081 HV Amsterdam, The Netherlands.}

\email{a.fernandezjimenez@vu.nl}

\keywords{Cahn-Hilliard, deterministic particle approximation, Wasserstein gradient flows, nonlocal Cahn-Hilliard equations, aggregation-diffusion}
\subjclass[2020]{35A01, 
35A15, 
35G20, 
35B36, 
49Q22} 

\begin{document}

\maketitle

\begin{abstract}
    We provide a deterministic particle approximation to a fourth order equation with applications in cell-cell adhesion. In order to do that, first we show that the equation can be asymptotically obtained as a limit from a class of well-posed nonlocal partial differential equations. These latter have the advantage that the particles’ empirical measure naturally satisfies the equation. Afterwards, we obtain stability of the $2$-Wasserstein gradient flow of this family of nonlocal equations that we use in order to recover a deterministic particle approximation of the fourth order equation. To our knowledge, in this manuscript we derive the first deterministic particle approximation for a fourth-order partial differential equation. Finally, we give some numerical simulations of the model at the particles level.
\end{abstract}

\section{Introduction}

In this manuscript, we are interested in the deterministic particle approximation (DPA) of the fourth-order equation with backward diffusion
\begin{equation}\label{eq:local CH}
    \p_{t}\rho + \dive(\rho\nabla\Delta\rho) +\Delta \rho^m=0    
\end{equation}
on the $d$--dimensional torus $\T^d=\R^d\setminus \mathbb{Z}^d$ with $|\T^d|=1$. Here $m > 1$ although for the convergence results, Theorems~\ref{thm:Nonlocal to local} and~\ref{thm:Particle approximation}, we only consider the case $m=2$. 
In \cite{Carrillo_Esposito_Falco_FJ23}, Carrillo, Esposito, Falc\'o and the second author understand the problem as a $2$-Wasserstein gradient flow.
Using a JKO scheme, the authors show existence of solutions for some values of $m$. More recently, in \cite{Buttenschon_Sinclair_Edelstein-Keshet24} the authors study the stability for $m = 2$ in dimension $d=1$. Moreover, there exists further literature focused on the mathematical analysis of some related equations, mostly of Cahn-Hilliard and thin-film type, c.f.~\cite{Elliott_Garcke96, Witel_Bern_Bert_EJAM04, Slepcev09, Lisini_Matthes_Savare12,Liu_Wang17, Parsch25}, with possible applications in lubrication theory, cf.~\cite{HOCHERMAN_ROSENAU_93,Bertozzi_Pugh_Nonlinearity_94,Otto_CPDE98,Ber98,DalPasso_Giacomelli_Shishkov_CPDE01,Grun_CPDE04} and the references therein. 

In addition, this family of equations also models tissue growth and patterning due to cell-cell adhesion \cite{Armstrong71, Duguay_Foty_Steinberg03}. In particular, in \cite{Falco_Baker_Carrillo22}, Falc\'o, Baker and Carrillo show that a $2$ species system version of \eqref{eq:local CH} keeps cell-sorting phenomena; see also \cite{Falco_Baker_Carrillo25} for further details.  However, there exists further literature where the authors suggest that nonlocal equations can be used to describe the cell-cell adhesion phenomenon \cite{Armstrong_Painter_Sherratt06, Carrillo_Murakawa_Sato_Togashi_Trush19}. Therefore, one of the goals of this manuscript is to rigorously bridge the two approaches by obtaining the local equation as an asymptotic of a nonlocal equation. This shows from a mathematical point of view that we can use both the local and nonlocal models in order to understand this phenomenon. Furthermore, as a consequence of our analysis, we are able to go even deeper and provide a microscopic description of the phenomenon. In particular,  in this manuscript we present a deterministic particle approximation, i.e. a DPA. 


It is not uncommon in nature to observe biological models with some components that might be stochastic. This happens, for example for small copy numbers or rare events. On the other hand, we can also observe different types of situations such as cell-cell adhesion, large copy numbers, or fast reactions where nature behaves deterministically. In particular, these phenomena can be understood via a DPA, e.g. \cite{Falco_Baker_Carrillo22, Volkening_Abbott_Chandra_Dubois_Lim_Sexton_Sandstede20, Bailo_Carrillo_Gomez-Castro24}. Therefore, it is also relevant to understand these type of models.




Before stating our DPA for \eqref{eq:local CH}, let us explain how to recover it starting from the more general aggregation equation (which is a nonlocal equation)
\begin{equation}\label{eq:aggregation_equation}
    \p_t \rho - \dive(\rho\nabla V_{\eps}\ast\rho)=0,    
\end{equation}
where $V_{\eps}$ is an interaction kernel depending on a small parameter $\eps$ accounting for the range of interaction between the particles or cells. 
More precisely, this formulation hinges on deterministic approaches since \textit{particles are solutions}, i.e. the empirical measure
\begin{equation*}
    \rho^N (t) \coloneqq \f{1}{N}\sum_{i=1}^{N}\delta_{X_{i}(t)}
\end{equation*}
is a weak solution of \eqref{eq:aggregation_equation}, where $X_i (t)$ gives the position of the particles. For any $i=1,\ldots,N$, the positions $X_{i}$ solve the system of ODEs
\begin{equation*}
    \dot{X_i} (t)= -\f{1}{N}\sum_{j=1}^{N}\nabla V_{\eps}(X_i (t) -X_j (t)) = -\nabla V_{\eps}\ast\rho^N(t).
\end{equation*}
Moreover, if $V_\ee \rightarrow \delta_0$ as $\ee \rightarrow 0$ we recover the  porous medium equation
\begin{equation*}
    \partial_{t}\rho - \frac{1}{2}\Delta\rho ^2=0
\end{equation*}
as a limit of \eqref{eq:aggregation_equation} when $\ee \rightarrow 0$. 
Choosing the interaction kernel appropriately as a \textit{squared kernel}, that is $V_\ee = \omega_\ee \ast \omega_\ee$ for some kernel $\omega_\ee$, one can make this argument rigorous. This reasoning provides a microscopic description of the porous medium equation: a DPA. The result we explain here is due to Lions and Mas-Gallic in \cite{Lions_Mas-Gallic01} who studied the problem for the first time. Recently, there has been several further results in this direction for more general non-linear diffusion evolution in time equations \cite{BurgerEsposito23, Carrillo_Esposito_Wu23, doumic_Hecht_Perthame_Peurichard24, Craig_Jacobs_Turanova24,  Carrillo_Esposito_Skrzeczkowski_Wu24, Amassad_Zhou25, Carrillo_Elbar_Fronzoni_Skrzeczkowski25, Difrancesco_Iorio_Schmidtchen25,2025arXiv251203185P} and for certain cases of aggregation-diffusion equations \cite{Carrillo_Craig_Patacchini19, Daneri_Radici_Runa22}. 

Regarding fourth-order differential equations, the first nonlocal Cahn–Hilliard equation was derived by Giacomin and Lebowitz \cite{Giacomin_Lebowitz97, Giacomin_Lebowitz98} looking at a microscopic description. Furthermore, in \cite{Bertini_Landim_Olla97} Bertini, Landim and Olla derived a constant mobility Cahn-Hilliard model as a hydrodynamic limit from a stochastic Ginzburg-Landau model. Since then, there has also been several results regarding nonlocal to local convergence for various fourth-order models. The case of Cahn-Hilliard type equations with constant mobility is covered in \cite{Melchionna_Ranetbauer_Scarpa_Trussardi18, Davoli_Ranetbauer_Scarpa_Trussardi20, Davoli_Scarpa_Trussardi21a, Davoli_Scarpa_Trussardi21b, Abels_Hurm23}. Degenerate mobility with aggregation given by a fixed potential is studied in \cite{Elbar_Skrzeczkowski23, Elbar_Skrzeczkowski24, Elbar_Gwiazda_Skrzeczkowski_Swierczewska25}. This result is then extended to its corresponding cross-diffusion system in \cite{Carrillo_Elbar_Skrzeczkowski23}. Finally, in \cite{Davoli_Marino_Pietschmann24} the authors cover nonlocal to local convergence for a Cahn-Hilliard type cross-diffusion system with non-linear mobility. However, all these nonlocal equations are not stemming from a deterministic particle system as in~\eqref{eq:aggregation_equation}. Therefore, even though the nonlocal to local convergence problem for 
Cahn-Hilliard is to our knowledge well-studied,
there are no results in the literature providing a deterministic particle approximation for any fourth-order partial differential equation. In order to derive this DPA let us make the following observation. From a Taylor expansion we have that for symmetric $\omega_\ee$ which has only second moments on the diagonal and is compactly supported on the ball with radius $\ee$:
\begin{equation}\label{def:Bee}
    B_\ee[\rho] \coloneqq \f{\rho-\rho\ast\omega_{\eps}}{\eps^{2}} = - \Delta \rho + O(\eps)
\end{equation}
for $\omega_\ee$ symmetric and an aproximation of $\delta_0$ as $\ee \rightarrow 0$. Hence, we can approximate \eqref{eq:local CH} by
\begin{equation*}
    \p_t\rho -\dive\left(\rho\nabla\left(\f{\rho-\rho\ast\omega_{\eps}}{\eps^2}\right)\right) +\Delta\rho^m=0
\end{equation*}
as it is suggested in different nonlocal Cahn-Hilliard models in the literature. Nevertheless, this approximation contains the terms $\Delta \rho^2$ and $\Delta \rho^m$, which diffuses particle initial data to a continuous density, i.e., the empirical measure does not remain as a sum of Dirac deltas due to the smoothing effect of the porous medium equation. Therefore this nonlocal equation is not at the origin of a DPA, as the empirical measure cannot be a solution of the equation. In order to overcome this difficulty we suggest the following equation
\begin{equation}\label{eq:nonlocal introduction}
    \p_{t}\rho - \dive\left(\rho\nabla\left( \f{\rho\ast\widetilde{\omega}_{\eps}\ast\widetilde{\omega}_{\eps}-\rho \ast\omega_{\eps} \ast\widetilde{\omega}_{\eps}\ast\widetilde{\omega}_{\eps} }{\eps^2}\right)\right) +\f{m}{m-1}\dive(\rho\nabla\widetilde{\omega}_{\eps}\ast(\rho\ast\widetilde{\omega}_{\eps})^{m-1})=\eps^*\dive(\rho\nabla\rho\ast R_{\alpha})
\end{equation}
where $\widetilde{\omega}_\ee=\omega_{\widetilde{\ee}}$ corresponds to another kernel converging to $\delta_0$ at the same time as $\omega_\ee$ but slower. We take, $\ee \ll \widetilde\ee \ll \eps^*$. In particular, in view of our proof, we need 
\begin{equation}\label{def:widetilde ee}
    \ee \coloneqq o( \widetilde\ee^{\frac{d+6}{2}}) \quad \text{and} \quad \widetilde\ee \coloneqq o (\ee^\ast) \quad \text{as $\ee\to 0$} .
\end{equation}
Note the presence of a purely artificial term with parameters $\eps^*, \alpha$ and kernel $R_\alpha$. Indeed, we need to introduce a viscosity term $\ee^* \Delta \rho^2$ to recover the compactness of $H^1$. However, in order to preserve the DPA we take the nonlocal approximation of this viscosity term with a parameter $\alpha$. 
Therefore, our system has 4 parameters:
\begin{itemize}
\item $\varepsilon$: the main parameter, controlling the nonlocality of the fourth-order term.
\item $\widetilde{\varepsilon}$: a parameter introduced to derive a DPA; it must converge slower than $\varepsilon$.
\item $\varepsilon^*$: an artificial parameter added to get the necessary $H^1$ bounds when $\alpha = 0$; this term must vanish in the limit.
\item $\alpha$: a parameter introduced to deduce a DPA and to guarantee $H^1$ bounds on the solution, which are recovered as $\alpha \to 0$. Later, for $m=2$ we show that we can also take $\alpha = \alpha(\ee)$ converging to $0$ at the same time as $\ee$, see Section~\ref{sec:Convergence W2}.
\end{itemize}
We keep these notations, rather than already choosing the last three parameters in terms of $\ee$ as the computations and the role of each parameter will become clearer.
Formulation~\eqref{eq:nonlocal introduction} has the advantage that an empirical measure is a weak solution where 
the equation of motion for the particles is given by
\begin{align*}
    \dot{X}_{i}(t) & = - \f{1}{N}\sum_{j=1}^{N}\nabla W_\ee (X_i (t) - X_j (t)) + \frac{m}{m-1}\sum_{j=1}^{N} \nabla \widetilde{\omega}_{\eps}(X_i(t) - X_j(t)) \left(\f{1}{N}\sum_{k=1}^{N}\widetilde{\omega}_{\eps}(X_j (t) - X_k (t))\right)^{m-1} \\
    & \quad - \eps^*\f{1}{N} \sum_{j=1}^{N} \nabla R_{\alpha}(X_i(t)-X_j(t)) 
\end{align*}
where 
$$
    W_\ee = \f{\widetilde{\omega}_{\eps}\ast\widetilde{\omega}_{\eps} - \omega_\eps\ast\widetilde{\omega}_{\eps}\ast\widetilde{\omega}_{\eps}}{\eps^2}.
$$
At this point we realise that both \eqref{eq:local CH} and \eqref{eq:nonlocal introduction} have a gradient flow structure with respect to the $2$-Wasserstein distance $\mathcal{W}_2$ \cite{Villani09, Santambrogio15, Santambrogio17}. In particular, for \eqref{eq:local CH} we consider
the (extended) energy functional 
\begin{equation}\label{eq:Free energy} 
        \mf[\rho] \coloneqq \begin{cases}\frac{1}{2}\displaystyle\int_\Td|\nabla\rho(x)|^2\,\dx-\me_m[\rho], &\rho\in L^{m}(\Td),\, \nabla\rho \in L^2(\Td),\\
        +\infty, & \mbox{otherwise}.
    \end{cases}
\end{equation}
For \eqref{eq:nonlocal introduction}  we take into account the regularised functional
\begin{equation*}
    \mathcal{F}_{\ee,\alpha} [\rho]  \coloneqq \frac{1}{4} \mathcal{D}_\ee [\rho \ast \widetilde\omega_\ee ] - \me_m [\rho \ast \widetilde\omega_\ee] +\frac{\eps^*}{2}\me_2[\rho\ast R^{\f 12}_{\alpha}]
\end{equation*}
where $R_\alpha = R^{\f 12}_{\alpha}\ast R^{\f 12}_{\alpha}$ and
\begin{equation}\label{eq:def_diss}
    \mathcal{D}_\ee [\rho ] = \iint_{\Td\times\Td} \frac{\omega_{\ee} (y)}{\ee^2} | \rho(x) - \rho(x-y) |^2 \dx \dy \quad \text{and} \quad \me_m [\rho] = \frac{1}{m-1} \int_\Td \rho^m (x) \dx.
\end{equation}
%
In particular, the energy decreases in time along the solution of the associated gradient flow. Thanks to the gradient flow structure,  we show stability of the gradient 
flows as $\ee$ tends to $0$. There are several stability results of this nature in the literature covering various cases of non-linear diffusion \cite{Lions_Mas-Gallic01, Carrillo_Esposito_Wu23, Carrillo_Esposito_Skrzeczkowski_Wu24}. There are also results in this direction for aggregation-diffusion equations, such as \cite{Carrillo_Craig_Patacchini19, Craig_Elamvazhuthi_Haberland_Turanova23}, where the authors also introduce a deterministic particle approximation through a blob method.
From this approach and taking advantage of the $\lambda$-convexity of the regularised free-energy functional we are able to achieve a rigorous particle approximation that follows from the $\lambda$-stability (or contractivity) of Wasserstein gradient flows \cite{Ambrosio_Gigli_Savare08}. However, one can only achieve a qualitative result since the initial datum needs to be approximated sufficiently \cite[Theorem 1.4]{Craig_Elamvazhuthi_Haberland_Turanova23}. Thereby, quantitative results are left as an open problem.

Further related to particle methods, we refer the reader to Chertock's comprehensive review on deterministic particle methods \cite{Chertock17}. We also mention the seminal paper by Oelsch\"ager \cite{Oelschlager90}, where a stochastic particle approximation is proven for classical and positive solutions of the quadratic porous medium equation in $\Rd$ and for weak solutions in the one dimensional case. We also mention the much more recent results in \cite{Chen_Daus_Jungel19, Chen_Daus_Holzinger_Jungel21} covering systems. In \cite{Philipowski07}, the author deduces strong $L^1$-solutions of the quadratic porous medium equation from a stochastic mean field interacting particle system with the addition of a vanishing Brownian motion. Finally, there has been further research for the viscous porous medium in \cite{Oelschlager01, Morale_Capasso_Oelschlager05, Figalli_Philipowski08} and for the porous medium equation with fractional diffusion in \cite{Chen_Holzinger_Jungel_Zamponi22}, where the authors study these equations as a limit of a sequence of distributions of the solutions to nonlinear stochastic differential equations.

\subsection{Main results}
We solve three different research questions. The well-posedness of the nonlocal equation \eqref{eq:nonlocal introduction}, and when $m=2$: the nonlocal to local convergence as $\alpha,\ee\to 0$, and the stability of the gradient flow in the limit. First, we introduce our mollifying kernel.

\begin{defn}
    [Mollifying kernel]
    \label{def:Mollifying sequence}
    We say that $\omega_1:\R^d\to \R$ is an admissible mollifier whenever
    \begin{enumerate}[label=$\roman*$.]
        \item $\omega_1$ is smooth, in $W^{2,\infty}(\R^d)$, compactly supported in the unit ball and nonnegative.
        \item It is symmetric, i.e. $\omega_1 (x) = \omega_1 (-x)$.
        \item $\int_{\Rd} \omega_1 (y) \dy = 1$, $\int_{\Rd} y \omega_1 (y) \, dy = 0$ and $\int_{\Rd} y_i y_j \omega_1 (y) \, dy = \delta_{i,j} \frac{2}{d}$.
    \end{enumerate}
\end{defn}

Hence, from $\omega_1$ one can generate a mollifying sequence. This one is defined by $\omega_\ee (x) = \ee^{-d} \omega_1 (x / \ee)$ for $\ee > 0$ and 
In particular, for $\widetilde\omega_\ee$ we make the choice $\widetilde\omega_\ee = \omega_{\widetilde\ee}$ with $\widetilde\ee$ defined in~\eqref{def:widetilde ee}. For the artificial term $\ee^*\dive(\rho \nabla R_\alpha\ast\rho)$ we need another set of assumption from~\cite{Amassad_Zhou25}. This allows us to find a rate of convergence with respect to $\alpha$. We refer to $R$ as the vanishing viscosity kernel.

\begin{defn}
    [Vanishing viscosity kernel]
    \label{def:vanishing_sequence}
    We say that $R$ is an admissible kernel whenever:
    \begin{enumerate}[label=$\roman*$.]
    \item $R$ is nonnegative, in $W^{2,\infty}(\mathbb{R}^d)$ and has bounded second moments.
    \item $\displaystyle \int_{\mathbb{R}^d} R(x)\, dx = 1.$
    \item The Fourier transform $\widehat{R}$ is positive and satisfies
\begin{equation*}
\left\{
\begin{aligned}
    &\frac{1}{a} \le \widehat{R}(\xi) \le 1, && |\xi|\le 1, \\
    &\frac{a}{|\xi|^{2k}} \le \widehat{R}(\xi) \le \frac{b}{|\xi|^{k}}, && |\xi|>1,
\end{aligned}
\right.
\end{equation*}
\textit{for some } $k>0$.

\item The decomposition $R = R^{\frac 12} * R^{\frac 12}$, defined by $\widehat{R^{\frac 12}} \coloneqq (\widehat{R})^{1/2}$, satisfies
\begin{equation*}
|\nabla R^{\frac 1 2}(x)| \le (R^{\frac 1 2} * h)(x), \quad \text{a.e. } x\in \mathbb{R}^d
\end{equation*}
for some $h\in \mathcal{M}(\mathbb{R}^d)$, i.e. Borel measure with bounded total variation.
\item For any $0<\eta<\alpha$, define $R_{\alpha} = \alpha^{-d} R(\cdot/\alpha)$, 
$R_{\eta} = \eta^{-d} R(\cdot/\eta)$ and consider the intermediate scale $L_{\alpha,\eta}$ defined as
\begin{equation}
\label{eq:1.8}
R_{\alpha} = R_{\eta} * L_{\alpha,\eta}, 
\qquad \widehat{L_{\alpha,\eta}} \coloneqq \frac{\widehat{R_{\alpha}}}{\widehat{R_{\eta}}}.
\end{equation}
The kernel $L_{\alpha,\eta}$ so defined belongs to $L^1(\mathbb{R}^d)$ and satisfies for some constant $C>0$,
\begin{equation}
\label{eq:1.9}
\int_{\mathbb{R}^d} |x| |L_{\alpha,\eta}(x)|\, dx \le C\alpha.
\end{equation}

\item $k$ is large enough  so that both $R$ and $R^{\f 12}$ have bounded second derivatives.

\end{enumerate}
\end{defn}

 Moreover, let us specify our definition of Fourier transform as
\begin{equation*}
    \widehat{f} (\xi) \coloneqq \int_\Rd f(x) e^{-2 \pi i x \cdot \xi} \dx .
\end{equation*}
\begin{rem}
    An example satisfying all of these properties introduced in Definition \ref{def:vanishing_sequence} is the kernel 
    \begin{equation*}
        R(x)=C_{d,k}\,|x|^{\frac{k-2d}{2}}\,K_{\frac{k-2d}{2}}(|x|)
    \end{equation*}
    where $C_{d,k}$ is the normalising constant and $K$ is the Bessel function (also known as Basset functions) given by
    \begin{equation*}
        K_{\nu}(r) = \int_{0}^{\infty} t^{-\nu-1} e^{-t-\frac{r^{2}}{4t}} \,dt, \qquad r>0.
    \end{equation*}
    Furthermore, $k > \frac{d+2}{2}$ in order to fulfill the regularity assumption.
\end{rem}
We can reformulate~\eqref{eq:nonlocal introduction} as
\begin{align}\label{eq:Nonlocal CH}\tag{NL}
    \begin{dcases}
        \frac{\partial \rho}{\partial t} + \dive \left( \rho \bold v_{\ee,\alpha} [\rho] \right)=0,\\
        \bold v_{\ee,\alpha} [\rho] = \nabla \left( - B_\ee [\rho \ast \widetilde\omega_\ee \ast \widetilde\omega_\ee] + \frac{m}{m-1} \widetilde\omega_{\ee} \ast (\rho \ast \widetilde\omega_\ee)^{m-1} - \eps^* \rho \ast  R_{\alpha} \right)  , 
     \end{dcases}
\end{align}
in $(0,T) \times \Td$ with initial condition $\rho (0, \cdot ) = \rho_0$ and with the convention $\rho\ast R_{\alpha}=\rho$ when $\alpha=0$ as $R_\alpha\to\delta_0$. The nonlocal problem \eqref{eq:Nonlocal CH} contains three different mollifiers $R_\alpha$, $\widetilde\omega_\ee$ and $\omega_\ee$.  Furthermore, \eqref{eq:Nonlocal CH} corresponds (formally) to the $2$-Wasserstein gradient flow of the free-energy functional
\begin{equation}\label{eq:nonlocal_free_energy}
    \mf_{\ee, \alpha} [\rho] \coloneqq \frac{1}{4} \mathcal{D}_\ee [\rho \ast \widetilde\omega_\ee ] - \me_m [\rho \ast \widetilde\omega_\ee] + \frac{\eps^*}{2} \me_2 [\rho \ast  R^{\f 12}_\alpha],
\end{equation}
with $\Dee$ and $\me_m$ defined in \eqref{eq:def_diss}. 
In particular, the free-energy decreases along the solutions of~\eqref{eq:Nonlocal CH}, providing first a priori bounds. Following the same notation, we introduce the local version of \eqref{eq:Nonlocal CH}
\begin{align}\tag{CH}\label{eq:CH}
    \begin{dcases}
        \frac{\partial \rho}{\partial t} + \dive \left( \rho \bold v[\rho] \right)=0,\\
        \bold v[\rho] = \nabla \left(\Delta \rho + \frac{m}{m-1} \rho^{m-1} \right) ,
    \end{dcases}
\end{align}
set also on $(0,T) \times \Td$ with initial condition $\rho (0, \cdot ) = \rho_0$. 
Thereby, the first main goal of this manuscript is to show well-posedness for the nonlocal problem \eqref{eq:Nonlocal CH}. In order to do that, we start by introducing the notion of weak solution for both~\eqref{eq:Nonlocal CH} and~\eqref{eq:CH}. 

\begin{defn}[Weak solution]\label{def:Weak solution}
    We say that $\rho$ is a weak solution of the  problem \eqref{eq:Nonlocal CH} if
    \begin{itemize}
        \item $\rho \in L^{\infty} (0,T; L^{\infty} (\Td))$,
        \item for every test function $\varphi \in C_c^{\infty} ([0, T ) \times \Td )$ we have
    \begin{align*}
        - \int_0^{T} \int_{\Td} \rho \frac{\partial \varphi}{\partial t} \, \dx \, \dt - \int_{\Td} \rho_0 \varphi (0, x) \, \dx & =   \int_0^{T} \int_{\Td} \rho \bold v_{\ee,\alpha} [\rho ] \cdot \nabla \varphi  \dx  \dt .
    \end{align*}
    \end{itemize}    
\end{defn}

For the local problem we use the following notion of solution.

\begin{defn}[Weak solution]\label{def:Weak solution local}
    We say that $\rho$ is a weak solution of the  problem \eqref{eq:CH} with $m=2$ if
    \begin{itemize}
        \item $\rho \in L^{\infty} (0,T; H^{1}(\Td))\cap L^{2}(0,T; H^{2}(\T^d))$,
        \item for every test function $\varphi \in C_c^{\infty} ([0, T) \times \Td )$ we have
   \begin{align*}
        & - \int_0^T \int_\Td \rho \frac{\partial \varphi}{\partial t} \dx \dt -  \int_\Td \rho_0 \varphi (0, x) \dx = \frac{1}{2} \int_0^T \int_\Td ( \nabla \rho (t, x) \otimes \nabla \rho (t, x) ) : D^2 \varphi (t, x) \dx \dt \\
        & \hspace{20mm} \quad + \frac{1}{2} \int_0^T \int_\Td |\nabla \rho (t, x) |^2 \Delta \varphi \dx \dt + \frac{1}{2} \int_0^T \int_\Td \rho (t, x) \nabla \rho (t, x) \cdot \nabla \Delta \varphi (t, x) \dx \dt \\
        & \hspace{20mm} \quad+ 2\int_0^T \rho(t, x) \nabla\rho (t, x) \cdot \nabla \varphi (t, x) \dx \dt  .
    \end{align*}
    \end{itemize}    
\end{defn}
The solutions constructed in this manuscript can have higher regularity, and an improved weak formulation (i.e. the regularity on the test function can be relaxed). As this is not the main goal of the paper, we keep these notations for simplicity.
Hence, we are ready to introduce the results of this manuscript. 
First, we prove an existence result for \eqref{eq:Nonlocal CH}.
\begin{thm}
    [Existence for the nonlocal problem]\label{thm:Nonlocal Existence}
    Assume $\ee , \alpha > 0$, and $\rho_0 \in L^{\infty} (\Td)$ is nonnegative. Assume furthermore $m \in (1, + \infty )$ and the mollifying sequence is as in Definition~\ref{def:Mollifying sequence} and~\ref{def:vanishing_sequence}. Then, the problem \eqref{eq:Nonlocal CH} has a weak solution in the sense of Definition~\ref{def:Weak solution}.
\end{thm}
We show this result in Subsection \ref{sec:Viscosity limit}. 
Afterwards, we show uniqueness of solutions for \eqref{eq:Nonlocal CH}.
\begin{thm}
    [Uniqueness for the nonlocal problem]\label{thm:Nonlocal Uniqueness}
    Assume $\ee , \alpha > 0$, and $\rho_0 \in L^{\infty} (\Td)$ is nonnegative. Assume furthermore $m \in [2, + \infty )$ and the mollifying sequence is as in Definition~\ref{def:Mollifying sequence} and~\ref{def:vanishing_sequence}. Then, the problem \eqref{eq:Nonlocal CH} has a unique weak solution.
\end{thm}
We prove this result in Subsection \ref{sec:Nonlocal Uniqueness}. Once we have discussed the properties of \eqref{eq:Nonlocal CH} we are ready to prove convergence of solutions from \eqref{eq:Nonlocal CH} to \eqref{eq:CH} for the case $m = 2$. We first define the entropy term
\begin{equation*}
    \Phi [\rho] \coloneqq \int_\Td \rho (\log \rho - 1)  \dx .
\end{equation*}
The dissipation of the entropy yields bounds on the solution. This functional is often used in many contexts, such as the Cahn-Hilliard equation~\cite{Elbar_Skrzeczkowski23, Elliott_Garcke96, MR3448925}. Also, for simplicity, we assume that the initial condition has mass 1, as we intend to use the Wasserstein distance, but the proof can be adapted to any mass.

\begin{prop}
[Convergence of nonlocal to local for each parameter]\label{thm:Nonlocal to local H1}
    Assume $\rho_0\in L^{\infty}(\Td)$ is of mass 1, with bounded free energy and entropy, that is $\sup_{0<\ee,\alpha<1}\mathcal{F}_{\ee, \alpha}(\rho_{0})<+\infty$ and $\Phi(\rho_0)<+\infty$, $m=2$, and the mollifying sequence is as in Definition~\ref{def:Mollifying sequence} and~\ref{def:vanishing_sequence}. Let $\{ \rho_{\ee, \alpha} \}$ be a sequence of solutions of the nonlocal equation \eqref{eq:Nonlocal CH} as constructed in Theorem~\ref{thm:Nonlocal Existence} with initial condition $\rho_0$. Then, there exists a sequence $\alpha_j \rightarrow 0$ as $j \rightarrow \infty$ such that for every $\ee > 0$ fixed
    \begin{equation}
    \begin{split}
    &\rho_{\ee , \alpha_j} \rightharpoonup \rho_{\ee,0} \quad \text{weakly in $L^{1}((0,T)\times \Td)$},\\
    &\rho_{\ee , \alpha_j} \ast  R^{\f 12}_{\alpha_j}  \rightharpoonup \rho_{\ee, 0}\quad  \text{weakly in } L^2(0,T; H^1(\Td)), \\
    & \rho_{\ee , \alpha_j} \ast  R^{\f 12}_{\alpha_j}  \rightarrow \rho_{\ee, 0} \quad \text{strongly in } L^2((0,T)\times \Td), 
    \end{split}
    \end{equation}
    where $\rho_{\ee, 0}$ is a weak solution of \eqref{eq:Nonlocal CH} for $\alpha = 0$ in the sense of Definition~\ref{def:Weak solution}. Furthermore, there exists a further subsequence $\ee_k \rightarrow 0$ as $k \rightarrow \infty$ such that
    \begin{align}\label{eq:Strong H1 NL to CH alpha 0}
       &\rho_{\eps_k,0}\rightharpoonup\rho, \quad \text{weakly in $L^{1}((0,T)\times \Td)$} \\ 
       &\rho_{\ee_k , 0} \ast \omega_{\ee_k} \rightarrow \rho \quad \text{strongly in } L^2(0,T; H^1(\Td))
    \end{align}
    where $\rho \in L^{\infty}(0,T; H^{1}(\T^d)) \cap L^{2}(0,T; H^{2}(\T^d))$ is a weak solution of the local equation \eqref{eq:CH} with initial condition $\rho_{0}$ in the sense of Definition~\ref{def:Weak solution}.
\end{prop}

We prove this result in Subsection \ref{sec:NL to L}. This proposition is at the core of the proof of our main theorem, that is the nonlocal to local convergence when the parameters converge to zero at the same time. We are able to show convergence along a subsequence $(\ee_k, \alpha_k)$ where $\ee_k, \alpha_k \rightarrow 0$ simultaneously as $k \rightarrow \infty$. 

\begin{thm}
    [Convergence of the nonlocal equation]\label{thm:Nonlocal to local}
    Assume $\rho_0\in L^{\infty}(\Td)$ is of mass 1 with bounded free energy and entropy, that is $\sup_{0<\ee,\alpha<1}\mathcal{F}_{\ee, \alpha}(\rho_{0})<+\infty$ and $\Phi(\rho_0)<+\infty$, $m=2$, and the mollifying sequence is as in Definition~\ref{def:Mollifying sequence} and~\ref{def:vanishing_sequence}. Let $\{ \rho_{\ee, \alpha} \}$ be a sequence of solutions of the nonlocal equation \eqref{eq:Nonlocal CH} from Theorem~\ref{thm:Nonlocal Existence} with initial condition $\rho_0$. Then, there exists a subsequence $\eps_k\rightarrow 0$ and a rate function $r$ which is exponential such that for $\alpha_k=r(\eps_k)$
    $$
    \rho_{\eps_k,\alpha_k}\to \rho \quad \text{narrowly on $(0,T)\times\T^d$}
    $$
    where $\rho$ is a weak solution of the local equation \eqref{eq:CH} with initial condition $\rho_{0}$ in the sense of Definition~\ref{def:Weak solution}.  
\end{thm}
We present a sketch of the proof in Subsection~\ref{sec:Sketch proof} and a rigorous proof in Section~\ref{sec:Convergence W2}. We prove another result, exploiting the $\lambda$-convexity of $\mf_{\ee, \alpha}$ with respect to the $2$-Wasserstein distance in order to deduce a deterministic particle approximation. Nevertheless, before stating the main result, we mention in a remark another consequence of the $\lambda$-convexity connected to the uniqueness of the sequence $\rho_{\ee, \alpha}$.
\begin{rem}[Uniqueness through $\lambda$-contractivity]
    We can take advantage of the $2$-Wasserstein $\lambda$-convexity to recover uniqueness of the \eqref{eq:Nonlocal CH} for $m= 2$. As explained in Section~\ref{sec:Convexity}, the regularised free energy $\mf_{\ee. \alpha}$ is $\lambda$-convex. Thanks to this,  the weak solution $\rho_{\ee, \alpha}$ in Theorem~\ref{thm:Nonlocal Existence} is unique among absolutely continuous curves $\rho : [0,T] \rightarrow \mP (\Td)$ satisfying \eqref{eq:Nonlocal CH}.
\end{rem}

We introduce the deterministic particle approximation to the local problem \eqref{eq:CH}. 
\begin{thm}
    [Particle approximation to \eqref{eq:CH}]\label{thm:Particle approximation}
    Assume $\rho_0\in L^{\infty}(\Td)$ is of mass 1 with bounded free energy and entropy, that is $\sup_{0<\ee,\alpha<1}\mathcal{F}_{\ee, \alpha}(\rho_{0})<+\infty$ and $\Phi(\rho_0)<+\infty$, $m=2$, and the mollifying sequence is as in Definition~\ref{def:Mollifying sequence} and~\ref{def:vanishing_sequence}.  For any $t \in [0,T]$, $N \in \mathbb{N}$, the empirical measure $\rho_{\ee , \alpha}^N (t) = \frac{1}{N} \sum_{i=1}^N \delta_{X^i_{\ee, \alpha} (t)}$ is a weak solution to \eqref{eq:Nonlocal CH} with initial condition $\rho^{N}(0) =\frac{1}{N} \sum_{i=1}^N \delta_{X^{i,0}_{\ee, \alpha} }$ provided the particles satisfy the ODE system given by  
    \begin{align*}
        \dot{X}^{i}_{\ee, \alpha}(t)  & = - \f{1}{N}\sum_{j=1}^{N}\nabla W_\ee(X^i_{\ee, \alpha}(t) - X^j_{\ee, \alpha}(t))  + 2\f{1}{N}\sum_{j=1}^{N} \nabla \widetilde{\omega}_{\eps}\ast\widetilde{\omega}_\ee(X^i_{\ee, \alpha}(t)-X^j_{\ee, \alpha}(t)) \\
        & \quad - \eps^*\frac{1}{N}\sum_{j=1}^{N} \nabla R_{\alpha}(X^i_{\ee, \alpha}(t)-X^j_{\ee, \alpha}(t))
    \end{align*}
    and $X^{i}_{\ee, \alpha}(0) = X^{i,0}_{\ee, \alpha}$. Here 
    $$
        W_\ee = \f{\widetilde{\omega}_{\eps}\ast\widetilde{\omega}_{\eps} - \omega_\eps\ast\widetilde{\omega}_{\eps}\ast\widetilde{\omega}_{\eps}}{\eps^2}.
    $$
    Let us consider the sequence $\rho_k = \rho_{\ee_k , \alpha_k}$ from Theorem~\ref{thm:Nonlocal to local} and the empirical measure solution $ \rho_{\ee_k, \alpha_k }^N$. 
    Furthermore, suppose that, up to a subsequence, there exists $N = N(k) \rightarrow + \infty$ as $k \rightarrow \infty$ such that 
    \begin{equation}\label{eq:Initial Assumption}
        \lim_{k \rightarrow \infty}  e^{\lambda_{ k} t} \mathcal{W}_2 (\rho^{N (k)}_{\ee_k, \alpha_k} (0) , \rho(0))  = 0, \quad \, t \in [0,T],
    \end{equation}
    where
    \begin{equation*}
        \lambda_{\ee, \alpha } \simeq \left(\ee^{-2} \widetilde\ee^{-d-2} +  \widetilde\ee^{-d-2} + \ee^\ast \alpha^{-d-2}\right) \quad \text{and} \quad \lambda_k = \lambda_{\ee_k, \alpha_k}.
    \end{equation*}
    Then, $\rho^{N} = \rho_{\ee_k, \alpha_k}^{N(k)}$ fulfils
    $$
    \rho^{N} \to \rho \quad \text{narrowly on $(0,T)\times\T^d$}
    $$
where $\rho$ is a weak solution of~\eqref{eq:CH} with initial condition $\rho_{0}$.
\end{thm}

In view of Theorem~\ref{thm:Particle approximation} and the initial condition \eqref{eq:Initial Assumption}, it is possible to choose the initial particle discretisation such that $\mathcal{W}_2 (\rho^{N} (0), \rho (0) ) = O(1/N)$, for a given $\rho (0)  \in \mP^a (\Td)$. We consider $\rho^N (0)$ obtained by dividing the area below the graph of $\rho (0)$ into $N$ different regions with equal masses. Then, one can just take $N = o \left( e^{ \lambda_{k}} \right)$ to fulfil the hypothesis on the initial condition. In this manuscript, we do not provide sharp rates, and it remains an open question for future research. 

\subsection{Structure of the paper}
In Section~\ref{sec:Preliminaries} we set the notation and definitions that we use in this manuscript. Furthermore, we also introduce some key properties of our main mathematical objects. In Section~\ref{sec:NL} we focus on the nonlocal problem \eqref{eq:Nonlocal CH}. We study existence and uniqueness of weak solutions, proving Theorem~\ref{thm:Nonlocal Existence} and Section~\ref{thm:Nonlocal Uniqueness}. In Section~\ref{sec:Non-L to L}, we prove convergence of solutions of the nonlocal problem \eqref{eq:Nonlocal CH} to solutions of the local problem \eqref{eq:CH}; we show Theorem~\ref{thm:Nonlocal to local H1}. In Section~\ref{sec:Convergence W2} we show that we can take the limit in $\ee$ and $\alpha$ simultaneously using the $2$-Wasserstein distance and discuss Theorem~\ref{thm:Nonlocal to local}. In Section~\ref{sec:Convexity} we provide a deterministic particle approximation. In order to do that, we exploit the $\lambda$-convexity of the free-energy functional with respect to the $2$-Wasserstein metric, proving  Theorem~\ref{thm:Particle approximation}. Finally, in Appendix~\ref{sec:Simulations} we provide some numerical simulations of the model at the particle level.

\subsection{Notations and functional settings}

We denote by $L^{p}(\Td)$, $W^{m,p}(\Td)$ the usual Lebesgue and Sobolev spaces, and by $\|\cdot\|_{L^{p}}$, $\|\cdot\|_{W^{m,p}}$ their corresponding norms. As usual, $H^{s}(\Td) = W^{s,2}(\Td)$. When the norms concern a vector or a matrix, we still write $\|\nabla u\|_{L^{p}}$ or $\|D^{2} u\|_{L^{p}}$ instead of $\|\nabla u\|_{L^{p}(\Td; \Td)}$ and $\|D^2 u\|_{L^{p}(\Td;\R^{d^2})}$. $\text{Lip}(\Td)$ is the space of Lipschitz functions in the torus. We often write $C$ for a generic constant appearing in the different inequalities. Its value can change from one line to another, and its dependence to other constants can be specified by writing $C(a,...)$ if it depends on the parameter $a$ and other parameters. 

\section{Preliminaries}\label{sec:Preliminaries}

In this section, we focus on gaining a better understanding of the problem. In order to do that, we describe the key properties of the nonlocal operators appearing throughout the manuscript, and we analyse the structure of the corresponding free energy. In Subsection~\ref{sec:Wasserstein metric} we introduce the Wasserstein metric; Subsection~\ref {sec:Nonlocal operators} is devoted to nonlocal operators and their properties. Subsection \ref{sec:free-energy} focuses on the free-energy functional. Finally, in Subsection~\ref{sec:Sketch proof} we present a sketch of the proof of our main theorems.

\subsection{Optimal transport and Wasserstein metric}\label{sec:Wasserstein metric}

A key tool for the analysis is the Wasserstein metric (see also \cite{Ambrosio_Gigli_Savare08, Villani09, Santambrogio15} for further background and more details on the definitions and remarks in this subsection). This is a distance function in the space of probability measures.

\begin{defn}
    [$2$-Wasserstein distance]
    For $\mu, \nu \in \mP(\Td)$ (the space of probability measures), we define the $2$-Wasserstein distance between $\mu$ and $\nu$ as
    \begin{equation}\label{eq:W2}
        \mathcal{W}_2 (\mu , \nu) \coloneqq \min_{\gamma \in \Gamma (\mu , \nu)} \left\{ \int_{\Td \times \Td} c_2(x,y) \diff \gamma (x, y) \right\}^{\frac{1}{2}},
    \end{equation}
    where the transport cost on the torus is 
    $$
        c_2(x,y) = \inf_{k\in \mathbb{Z}^d}|x-y+k|^2,
    $$
    $\Gamma (\mu , \nu)$ is the set of transport plans between $\mu$ and $\nu$, 
    \begin{equation*}
        \Gamma (\mu , \nu ) = \left\{ \gamma \in \mP (\Td \times \Td) : (\pi_x)_\# \gamma = \mu, (\pi_y)_\# \gamma = \nu \right\},
    \end{equation*}
    and $\pi_x$ and $\pi_y$ are the projections onto the first and second variables, respectively.
\end{defn}

In the expressions above, $\mu$ and $\nu$ are the push-forward of $\gamma$ through the projections $\pi_x$ and $\pi_y$. For a measure $\mu \in \mP (\Td)$ and a Borel map $T : \Td \rightarrow \T^d$, $n \in \mathbb{N}$, the push-forward of $\mu$ through $T$ is defined by
\begin{equation*}
    \int_\Td f(y) \diff T_\# \mu (y) = \int_\Td f(T(x)) \diff \mu (x) \quad \text{for all Borel functions } f \text{ on } \Rn.
\end{equation*}

It is known that when $\mu\ll \mathcal{L}^d$, where $\mathcal{L}^d$ is the Lebesgue measure then the optimal transport plan $\gamma_0$ is induced by a transport map $T:\Td\to \Td$ satisfying $\gamma_0 = (Id, T)\sharp\mu$. Moreover, $T(x) = x-\nabla\varphi(x)$ where $\varphi$ is the so-called Kantorovich potential. In addition to this,  there exists $\psi$ such that $T^{-1}(x) = x-\nabla\psi$ and we have $(\text{id}-\nabla\varphi)\sharp \mu = \nu $ and $(\text{id}-\nabla\psi)\sharp \nu = \mu $.

Finally, we have the following lemma, which can be adapted from~\cite[Theorem 8.4.7]{Ambrosio_Gigli_Savare08} and~\cite{Amassad_Zhou25}. 

\begin{lem}\label{lem:action-minimising path}
    Let $\mu_0, \mu_1 \in L^\infty ([0,T] ; \mathcal{P} (\Td))$, absolutely continuous w.r.t. the Lebesgue measure. Assume that they solve the convection equations
    \begin{equation*}
        \partial_t \mu_i (t,x) + \dive (v_i (t,x) \mu_i (t,x) ) = 0, \quad i = 0,1,
    \end{equation*}
    in distribution, with
    \begin{equation*}
        \int_0^T \int_\Td \mu_i (t,x) | v_i (t,x) |^2 \dx \dt < \infty , \quad i = 0,1.
    \end{equation*}
    Then, the following formula holds for $\mathcal{L}^1-a.e$ $t\in(0,T)$:
    \begin{equation*}
        \frac{\diff}{\diff t} \left[ \frac{1}{2} \mathcal{W}_2^2 (\mu_0 (t, \cdot) , \mu_1 (t, \cdot) ) \right] = \int_\Td \nabla\varphi^t(x)  \cdot v_0(t,x) \mu_0(t,x) \diff x + \int_\Td \nabla\psi^t(x)  \cdot v_1(t,x) \mu_1(t,x) \diff x 
    \end{equation*}
    where $(\varphi^t,\psi^t)$ is any pair of Kantorovich potentials in the optimal transport problem of $\mu_0(t)$ to $\mu_1(t)$. In particular
    \begin{equation*}
       \mathcal{W}_2^2(\mu_0 (t, \cdot) , \mu_1 (t, \cdot) ) = \int_\Td \mu_0|\nabla\varphi^t|^2\diff x.
    \end{equation*}
    Furthermore, there exists some constant $C$ such that
    \begin{equation*}
        \sup_{t\in[0,T]} ( \| \nabla\varphi^t\|_{L^\infty (\Td)} + \| \nabla\varphi^t  \|_{BV(\Td)} ) \leq C. 
    \end{equation*}
\end{lem}

\begin{rem}
The previous bounds do not hold on an unbounded domain. Indeed  $\nabla\varphi^t(x)= x-T^t(x)$ where $T^t$ is the transport map from $\mu_0(t)$ to $\mu_1(t)$, and we can immediately see on the torus that this quantity is bounded in $L^{\infty}$. The $BV$ bound can easily be deduced by using the well-known fact that $D^2\varphi\le I_d$. Such a reasoning does not hold on $\R^d$; for instance, $x-T(x)$ can be large if mass is transported very far. This is the main reason to work on the torus. 
\end{rem}

Another notion that is going to be relevant to us is the one of $\lambda$-convexity. In the Eulerian setting we say that a function $F$ is $\lambda$-convex if $F(x) + \frac{\lambda}{2} |x|^2$ is convex. Let us now consider the gradient flow problem $\dot{x}(t) = - \nabla F (x(t))$ with initial datum $x(0) = x_0$. 
Then, for two solutions $x_1, x_2$ it follows that
\begin{align*}
    \frac{\diff}{\dt} \frac{1}{2} | x_1 (t) - x_2(t)|^2 & =  (x_1 (t) - x_2(t) ) \cdot ( x_1'(t) -x_2'(t) )  = -  ( x_1 (t) - x_2(t) ) \cdot ( \nabla F(x_1(t)) - \nabla F(x_2(t)) ) \\
    & \leq \lambda | x_1 (t) - x_2 (t) |^2.
\end{align*}
Thus, from Gr\"onwall's inequality, we have that
\begin{equation*}
    |x_1 (t) - x_2 (t) | \leq e^{\lambda t} | x_1(0) - x_2(0)|
\end{equation*}
which, for instance, implies uniqueness of solutions. This argument can be extended to $(\mathcal{P} (\Td) , \mathcal{W}_2 )$. In order to do that, we first need to introduce the notion of $\lambda$-convexity along constant speed geodesic curves \cite{McCann97}. 
We say that a curve $\gamma : [0,1] \longrightarrow \mathcal{P} (\Td)$ is a \textit{constant speed curve} if
\begin{equation*}
    \mathcal{W}_2 (\gamma_s, \gamma_t) = (t-s) \mathcal{W}_2 (\gamma_0 , \gamma_1) \quad \text{for all } 0 \leq s \leq t \leq 1.
\end{equation*}
Thus, we can introduce the definition of $\lambda$-geodesically convex functionals
\begin{defn}
    For $\lambda \in \R$ (positive or negative), 
    we say $\mathcal{F}$ is $\lambda$-geodesically convex if for any $\nu_0, \nu_1 \in \mathrm{Dom} (\mathcal{F})$ there exists a constant speed geodesic $\gamma$ with $\gamma_0 = \nu_0$, $\gamma_1 = \nu_1$ such that 
    \begin{equation*}
        \mathcal{F} [\gamma_t] \leq (1-t) \mathcal{F}[\gamma_0] + t \mathcal{F} [\gamma_1] - \frac{\lambda}{2} t(1-t) \mathcal{W}_2(\gamma_0 , \gamma_1)^2 \quad \text{for all } t \in [0,1].
    \end{equation*}
    Furthermore, we say $\mf$ is displacement convex if $\lambda \geq 0$.
\end{defn}

Therefore, we can extend the Eulerian problem in order to study the problem $\partial_t \rho = - \nabla_{\mathcal{W}_2} (\mathcal{F} [\rho ])$ with initial datum $\rho (0) = \rho_0$. Hence, if $\mathcal{F}$ is $\lambda$-geodesically convex we have that for two solutions $\rho_1, \rho_2$ we have that
\begin{equation}\label{eq:gronwall_lambda_convex}
    \mathcal{W}_2 (\rho_1(t) , \rho_2(t) ) \leq e^{-\lambda t} \mathcal{W}_2 (\rho_1(0) , \rho_2(0) ),
\end{equation}
see \cite[Theorem 11.2.1]{Ambrosio_Gigli_Savare08}. This result is very fruitful since it helps to understand uniqueness, stability, or the long time behaviour among others.\\

The following lemma can be found in~\cite[Chapter 9]{Ambrosio_Gigli_Savare08} (see also \cite[Theorem 2.1]{Carrillo_McCann_Villani03}).
\begin{lem}[$\lambda$-geodesic convexity of the interaction energy]\label{lem:lambda_convex_interaction}
Let $\lambda\le 0$ and $W:\R^d\to \R^d$ be a $\lambda$-convex function with $W(-x)=W(x)$ for every $x\in\R^d$. Then the functional
$$
\mu\mapsto  \frac{1}{2}\int_{\T^d}\int_{\T^d}W(x-y)\diff \mu(x) \diff \mu(y),\quad \text{for every $\mu\in \mathcal{P}(\Td)$}
$$
is $\lambda$ geodesically convex.
\end{lem}

\subsection{The nonlocal operators}\label{sec:Nonlocal operators}
In this project, we will work with several nonlocal operators. Here, we present them and study some of their properties, bridging them with their local analogues. First, we recall the definition of $B_{\eps}$ from~\eqref{def:Bee}:
$$
B_{\ee}[f] = \int_{\T^d}\frac{\omega_{\ee}(y)}{\ee^2}(f(x)-f(x-y)).
$$

Then, we introduce the operator
\begin{equation}\label{def:S}
    S_\ee [f] (x,y) \coloneqq \frac{\sqrt{\omega_\ee (y)}}{\sqrt{2} \ee} (f(x-y) - f(x)),
\end{equation}
where $\omega_\ee$ is the mollifying sequence from Definition~\ref{def:Mollifying sequence}. We can use this operator to approximate the gradient.
For the sake of completeness, we also comment on some of its properties already discussed in \cite[Lemma 3.4]{Elbar_Skrzeczkowski23}.

\begin{lem}
    The operator $S_\ee$ satisfies:
    \begin{enumerate}[label=(S$_{\arabic*}$)]\label{lem:S properties} 
    \item 
    \label{S linear}  
    $S_\ee$ is a linear operator that commutes with derivatives with respect to $x$.

    \item 
    \label{S (ii)}
    For all functions $f, g : \Td \rightarrow \R$ we have
    \begin{align*}
        S_\ee [fg] (x,y) & - S_\ee [f] (x,y) g (x) - S_\ee [g] (x,y) f(x) = \\
        & \frac{\sqrt{\omega_\ee (y)}}{\sqrt{2} \ee}[(f(x-y) - f(x)) (g(x-y) - g(x))].
    \end{align*} 

    \item 
    \label{S integration by parts}
    For all $f, g \in L^2 (\Td )$
    \begin{equation*}
        \left\langle B_\ee [f] (\cdot) , g(\cdot ) \right\rangle_{L^2(\Td)} = \left\langle S_\ee [f] (\cdot , \cdot) , S_\ee [g] (\cdot , \cdot) \right\rangle_{L^2(\Td \times \Td)}  
    \end{equation*} 
    with $B_\ee$ defined in \eqref{def:Bee}.

    \item 
    \label{S convergence}
    If $\left\lbrace f_\ee \right\rbrace$ is strongly compact in $L^2(0,T ; H^1 (\Td ))$ and $\varphi \in L^{\infty} ((0,T) \times \Td )$ we have as $\ee\to 0$ (up to a subsequence not relabeled)
    \begin{equation*}
        \int_0^T \int_\Td \int_\Td (S_\ee [f_\ee] )^2 \varphi \rightarrow  \int_0^T \int_\Td| \nabla f |^2 \varphi
    \end{equation*}
    where $f$ is the strong limit of $f_\ee$.
\end{enumerate}
\end{lem}

This operator corresponds to a nonlocal approximation of the gradient, and it will help us to get \textit{a priori} bounds.

\begin{lem}
\label{lem:nonlocal operators}
 The  gradient and Laplacian can be bounded by their nonlocal counterparts in the following way: there exists $C$ such that for all $f\in L^{1}(\Td)$ depending only on the $L^{1}(\Td)$ norm of $f$ such that:
    \begin{align}
        \| \nabla (f \ast \widetilde\omega_\ee) \|_{L^2(\Td)}
        ^2& \leq C\iint_{\T^d\times\T^d} \left|S_\ee [f \ast \widetilde\omega_\ee] \right|^2 \dy \dx  +  C\frac{\ee^2}{\widetilde\ee^{d+4}} , \label{eq:Lr NL Gradient} \\
        \| \Delta (f \ast \widetilde\omega_\ee ) \|_{L^2 (\Td)}^2 & \leq C
        \iint_{\T^d\times\T^d} \left|S_\ee [\nabla (f \ast \widetilde\omega_\ee)] \right|^2 \dy \dx   +C\frac{\ee^2}{\widetilde\ee^{d+6}} ,  \label{eq:L2 NL Laplacian} \\
        \|D^2f\ast\widetilde{\omega}_{\ee}\|_{L^{2}} &\le \|D^2\widetilde{\omega}_\ee\|_{L^{2}}\|f\|_{L^{1}}\le \frac{C}{\ee^{\frac{d}{2}+2}}.
    \end{align}
\end{lem}

\begin{rem}
Note that $\iint_{\T^d\times\T^d} \left|S_\ee [f \ast \widetilde\omega_\ee] \right|^2 \dy \dx=\frac{1}{2}\mathcal{D}_\ee [ f \ast \widetilde\omega_\ee ]$ where $\mathcal{D}_\ee$ is defined in~\eqref{eq:def_diss}. 
\end{rem}

\begin{proof}
    Let us start by performing a Taylor expansion $f\ast\widetilde{\omega}_\ee(x-y) = f\ast\widetilde{\omega}_\ee(x) -y\cdot\nabla f\ast\widetilde{\omega}_{\eps}(x) + \f{1}{2}y^T D^2(f\ast\widetilde{\omega}_{\eps})(\theta)y$ for some $\theta \in [x,y]$. Then the proof follows by estimating $\|D^2f\ast\widetilde{\omega}_{\ee}\|_{L^{2}}\le \|D^2\widetilde{\omega}_\ee\|_{L^{2}}\|f\|_{L^{1}}\le \frac{C}{\ee^{\frac{d}{2}+2}}$ and the fact that $\omega_{\eps}$ is of zero average and compactly supported in a ball with radius $\ee$. The proof of the second inequality is similar, note that we do not recover estimates on the full Hessian because only the diagonal terms are seen by the kernel, which satisfies Definition~\ref{def:Mollifying sequence} (but they can be found by elliptic regularity of the Laplacian). 
\end{proof}

In the local case, the Gagliardo-Nirenberg inequality plays a key role in the analysis of the problem; see \cite{Carrillo_Esposito_Falco_FJ23}. Here we recall this inequality, and we introduce a nonlocal approximation of it that will appear repeatedly.

\begin{lem}[Gagliardo-Nirenberg interpolation inequality]
    Let $\theta\in[0,1]$ and  $1\le p,q\le +\infty$ such that $\frac{1}{p}=\theta\left(\frac{1}{2}-\frac{1}{d}\right)+\frac{1-\theta}{q}$. Then, it holds
    \begin{itemize}
        \item Local:
            \begin{equation}\label{eq:GN}
                \|f\|_{L^p(\Td)}\le C\|\nabla f\|_{L^2(\Td)}^\theta\|f\|_{L^q(\Td)}^{1-\theta} .
            \end{equation}
        \item Nonlocal:
            \begin{equation}\label{eq:NL GN}
                \| f \ast \widetilde\omega_\ee \|_{L^p (\Td)} \leq C  \left( \left(\iint_{\T^d\times\T^d} \left|S_\ee [f \ast \widetilde\omega_\ee] \right|^2 \dy \dx + C\frac{\ee^2}{\widetilde\ee^{d+4}} \right) ^{\frac{\theta}{2}}     \| f \ast \widetilde\omega_\ee \|_{L^q(\Td)}^{1-\theta} \right).
            \end{equation}
    \end{itemize}
    where $C$ denotes a positive constant depending on $p,q$ and $\|f\|_{L^1(\T^d)}$. In the case $d=2$, $\theta\in[0,1)$.
\end{lem}

\begin{proof}
The nonlocal Gagliardo-Nirenberg follows from combining  \eqref{eq:Lr NL Gradient} with \eqref{eq:GN}.
\end{proof}

\normalcolor

\subsection{Properties of the energy functional}\label{sec:free-energy}
The nonlocal Gagliardo-Nirenberg inequality allows us to recover uniform bounds from the free-energy $\mathcal{F}_{\ee, \alpha }$, analogously to the local case already discussed in \cite{Carrillo_Esposito_Falco_FJ23}. Let us  set
\[
2^{\ast}\coloneqq\begin{cases}
    +\infty &\mbox{if } d=1,2,\\
    \frac{2d}{d-2} & \mbox{if } d\ge3.
\end{cases}
\]
Hence, in the following proposition we include first \textit{a priori} estimates. 

\begin{prop}[Regularity results]\label{prop:Some properties}
    Assume $\rho \in L^1 (\Td)$ is nonnegative and let $1 < m < m_c:= 2 + \frac{2}{d}$. The following properties hold.
    \begin{enumerate}
        \item \underline{Lower bound for the free energy}: Let $\Dee[\rho \ast \widetilde{\omega}_\ee] <\infty$. Then there exists $\beta>0$ such that $\mathcal{F}_{\ee, \alpha } [\rho]$ is bounded from below as
        \begin{equation}\label{eq:Free energy bdd below}
            \mathcal{F}_{\ee, \alpha } [\rho] \geq - C\left(1 + \frac{\ee^2}{\widetilde\ee^{d +4}}  \right)^{\beta},
        \end{equation}
        where $C>0$ depends on $\|\rho\|_{L^{1}(\Td)}$ and is independent of $\ee$. 
        \item \underline{$H^1$-bound}: Assume $\mathcal{F}_{\ee , \alpha} [\rho] < \infty$, then $\Dee[\rho \ast \widetilde{\omega}_\ee] <\infty$ and
        \begin{align}
            \| \nabla (\rho \ast \widetilde\omega_\ee ) \|_{L^2 (\Td)} & \leq C \left(  \Dee[\rho \ast \widetilde{\omega}_\ee] + \frac{\ee^2}{\widetilde\ee^{d + 4} }\right)^{\frac{1}{2}} \label{eq:Grad convolved unif}.
        \end{align}
        Note that by definition of $\Dee[\rho \ast \widetilde{\omega}_\ee]$, we also have
        $$
        \|S_{\eps}[\rho\ast\widetilde{\omega}_{\eps}]\|_{L^2(\T^d\times\T^d)}\le C.
        $$
        \item \underline{$L^p$-regularity}: Assume $\mathcal{F}_{\ee , \alpha} [\rho]  < \infty$, then $\rho \ast \widetilde\omega_\ee \in L^p (\Td)$ for any $p\in[1,2^*]$, $d\neq 2$, and for any $p\in[1,2^*)$ when $d=2$. In particular, 
            \begin{equation}\label{eq:Convolution Lp}
                \| \rho \ast \widetilde\omega_\ee \|_{L^p (\Td)} \leq C.
            \end{equation}
            Note that $1 < m < 2 + \frac{2}{d} < 2^{\ast}$. 
    \end{enumerate}
\end{prop}

\begin{proof}
    The proof is simply a consequence of~\eqref{eq:NL GN} and Young's inequality. As it is completely similar to the local case in~\cite[Proposition 3.1]{Carrillo_Esposito_Falco_FJ23}, we do not repeat the proof here.
\end{proof}

\subsection{Sketch of proof}\label{sec:Sketch proof}
Our first result is the existence and uniqueness of solutions to the system~\eqref{eq:Nonlocal CH}. 
The equation has the structure 
\[
\partial_t \rho + \dive(\rho\, \bold{v}_{\ee, \alpha}[\rho])=0,
\]
where $\bold{v}_{\ee, \alpha}[\rho]$ is smooth, given by a convolution with smooth kernels. 
Such equations can be seen as aggregation equations with smooth interaction kernels, 
for which several existence results are available. 
Here we adopt a simple approach based on the viscosity approximation combined with a fixed point argument.  

More precisely, for fixed smooth $v$, the equation
\[
\partial_t\rho - \nu \Delta \rho + \dive(\rho\nabla v) = 0
\]
has smooth solutions. 
The Schauder fixed point theorem then yields solutions of
\[
\partial_t\rho - \nu \Delta \rho + \dive(\rho\nabla \bold{v}_{\ee, \alpha}[\rho]) = 0.
\]
Letting $\nu \to 0$, we pass to the limit thanks to weak compactness in $L^2$ 
and the strong convergence of $\bold{v}_{\ee, \alpha}[\rho]$ (again due to the convolution structure). 
This proves the existence of weak solutions, Theorem~\ref{thm:Nonlocal Existence}. 
Moreover, in regimes where $\bold{v}_{\ee, \alpha}[\rho]$ is Lipschitz (for instance when $m\geq 2$), 
we can also prove uniqueness, that is Theorem~\ref{thm:Nonlocal Uniqueness}. 
The argument is based on the $H^{-1}$ norm of the difference between two solutions 
and applying Gronwall’s lemma. 
This approach is standard in the study of aggregation–diffusion equations (see, e.g. \cite{Bertozzi_Slepcev10}).  

\medskip

Our second main result is the convergence from the nonlocal to the local model. 
For clarity, let us temporarily set aside the artificial regularization term involving $\alpha$. 
Without this additional term, we are not able to prove convergence. 
The main obstacle is that the passage to the limit involves $H^1$ estimates on the solutions, 
which we cannot obtain, even depending badly on $\ee$.  

To overcome this, we add a nonlinear diffusion of the form $\frac{\eps^*}{2}\Delta \rho^2$. 
From the entropy dissipation, that is computing $\frac{\diff}{\dt}\int_{\Td}\rho\log\rho$, one gets bounds of the type
\[
\eps^* \|\nabla \rho
\|_{L^2} \leq C.
\]
Even if the bound depends badly on $\varepsilon$, this is enough to prove convergence as long as $\eps\ll \eps^*$. 
However, simply adding $\eps^*\Delta\rho^2$ directly into the nonlocal equation 
is not relevant for our purposes, 
since our goal is to derive a DPA of a fourth–order equation, 
as explained in the introduction.  

A natural alternative is to introduce instead the regularisation
\[
\eps^*  \dive\big(\rho \nabla \rho \ast R_\alpha\big).
\]
Sending $\alpha \to 0$ yields the term $\frac{\eps^*}{2} \Delta\rho^2$. 
One could then let $\varepsilon \to 0$ with $\eps^*$ depending on $\varepsilon$, 
but this is a double limit procedure as we explain in Proposition~\ref{thm:Nonlocal to local H1}. 
In general, we expect to rely only on a single limit depending on $\varepsilon$, Theorem~\ref{thm:Nonlocal to local}.  

Our strategy is therefore to keep the nonlocal regularization 
\[
\eps^*\dive\big(\rho \nabla\rho \ast R_\alpha\big)
\]
and to quantify the convergence rate in terms of $\alpha$. 
More precisely, generalizing the method of \cite{Amassad_Zhou25}  in high dimensions (see also \cite{Carrillo_Elbar_Fronzoni_Skrzeczkowski25}  in one dimension) to the aggregation diffusion case, 
we obtain an estimate in the Wasserstein distance of the form
\[
\mathcal{W}_2(\rho_{\varepsilon,\alpha}, \rho_{\varepsilon, 0}) \leq C\frac{\alpha^{r_1}}{\eps^{r_2}} \exp{ \left( {\frac{T}{\eps^{r_2}}} \right)},
\] 
for suitable rates $r_1, r_2$.
Here, $\rho_{\varepsilon, 0}$ denotes the solution of the limiting system with $\alpha=0$, 
which includes the term $\frac{\eps^*}{2} \Delta \rho^2$.  

Finally, let $\rho_{\varepsilon,\alpha}$ be the sequence of solutions to the original nonlocal problem. 
For $\alpha=0$, i.e. in the diffusion case with $\frac{\eps^*}{2} \Delta \rho^2$, 
one can choose $\eps^*$ depending on $\varepsilon$ such that a subsequence $\rho_{\varepsilon_k,0}$ 
converges to some limit $\rho$ in some weak topology. 
Combining this with the rate in $\alpha$, we recover for smooth $\varphi$
\begin{align*}
    \int_{0}^{T}\int_{\T^d}\varphi(\rho_{\varepsilon_k,\alpha_k}- \rho) & \leq \int_{0}^{T}\mathcal{W}_2(\rho_{\varepsilon_k,\alpha_k}(t, \cdot), \rho_{\varepsilon_k,0}(t, \cdot))  + \int_{0}^{T}\int_{\T^d}\varphi(\rho_{\varepsilon_k,0}- \rho)  \\
    & \leq C\, \frac{\alpha_k^{r_1}}{\varepsilon_k^{r_2}}\exp{\left(\frac{T}{\eps_k^{r_2}}\right)} +\int_{0}^{T}\int_{\T^d}\varphi(\rho_{\varepsilon_k,0}- \rho).
\end{align*} 
By choosing $\alpha_k$ appropriately in terms of $\varepsilon_k$, 
both terms vanish as $k\to\infty$, which concludes the proof.  

Finally, taking advantage of this last result, we are able to recover a deterministic particle approximation relying on the \textit{blob method} and the $\lambda$-convexity of the regularised free energy, cf~Theorem~\ref{thm:Particle approximation}. 

\section{Nonlocal problem: Well-posedness for fixed \texorpdfstring{$\ee , \alpha >0$}{epsilon, alpha > 0}}\label{sec:NL}
The aim of this section is to discuss the well-posedness of the nonlocal problem~\eqref{eq:Nonlocal CH} for fixed $\eps, \alpha >0$. We prove existence when $m > 1$, Theorem~\ref{thm:Nonlocal Existence}, and uniqueness when $m \geq 2$, Theorem~\ref{thm:Nonlocal Uniqueness}. We note that the proof can be easily adapted to the whole space $\R^d$ (assuming also that the initial condition has a bounded second moment).

The problem \eqref{eq:Nonlocal CH} is a gradient flow with respect to the $2$-Wasserstein metric $\mathcal{W}_2$. Therefore, a possible proof of existence would be to use the celebrated JKO scheme, introduced in \cite{Jordan_Kinderlehrer_Otto98} by Jordan, Kinderlehrer, and Otto, to study 
the Fokker-Planck equation from a gradient flow perspective. For example, existence for the local problem \eqref{eq:local CH} is shown via a JKO scheme in \cite{Carrillo_Esposito_Falco_FJ23}. Here, we decide to approach the problem from a fixed point argument,  which we consider very natural since the kernels are smooth. With respect to the uniqueness problem, we choose a Gr\"onwall type argument.
Since we only cover the nonlocal problem during this section, we will use the notation $\rho \coloneqq \rho_{\ee, \alpha}$ when we refer to the solutions of the problem \eqref{eq:Nonlocal CH}.

\subsection{Existence for a regularized system}\label{sec:Viscosity term}

We add an artificial viscosity term $-\nu\Delta\rho$ and first prove existence for the following system
\begin{equation}\label{eq:Regularised Problem}\tag{P$_{ \nu}$}
  \frac{\partial \rho}{\partial t} - \nu\Delta\rho +  \dive \left( \rho  \bold v_{\ee,\alpha}[\rho] \right) =0
\end{equation}
on $(0,T)\times \Td$ and with $\rho(0,\cdot) = \rho_{0}$ where $\rho_0\in L^{\infty}(\Td)$ with $\rho_{0}\ge 0$, $\|\rho_{0}\|_{L^{1}} =1$ (without loss of generality) and $\nu>0$ fixed. We state the following result.
\begin{prop}
    [Existence for the viscosity problem]\label{thm:Viscosity term Existence}
    Assume $\nu, \ee, \alpha > 0$, and $\rho_0 \in  L^{\infty} (\Td)$. Then, the problem \eqref{eq:Regularised Problem} has a classical solution on $(0,T)\times\T^d$.
\end{prop}

In fact, $\rho_0$ is only in $L^{\infty}(\T^d)$, so on $t\in[0,T]$ our solution is weak, but becomes smooth and classical for $t>0$ by standard parabolic regularization.  Before proving this statement, let us recall some properties of the Fokker-Planck equation \begin{equation}\label{eq:system_rewritten_boldv}
        \p_t \rho - \nu \Delta\rho + \dive(\rho v)=0.
    \end{equation}

The problem \eqref{eq:system_rewritten_boldv} with initial condition $\rho_{0}\in L^1(\Td) \cap L^{\infty} (\Td)$, $v$ smooth enough, and $\nu > 0$ has a unique classical solution. This solution also satisfies the Maximum Principle. The literature concerning this problem is extensive; we refer for instance to \cite{LSU68}. We also obtain some \textit{a priori} bounds.

\begin{lem}[Estimates on the viscous equation]\label{lem:est_nu}
    Assume $\rho_{0}\in L^{\infty}(\Td)$, $v$ given and smooth enough, and $\nu > 0$. Then, the unique classical solution $\rho$ of the problem \eqref{eq:system_rewritten_boldv} is such that for all $p \in [1, \infty]$ (with convention $\frac{p-1}{p}=1$ when $p=\infty$) it follows that:
    \begin{align*}
& \int_{\Td}\rho(t) = \int_{\Td}\rho_0,\quad \text{for all $t\in[0,T]$} ,\\
&\| \rho \|_{L^{\infty} (0,T; L^p(\Td))} \leq \|\rho_0 \|_{L^p(\Td)} \exp{\left( \frac{(p-1) \| v \|_{L^{\infty}(0,T; W^{1,\infty}( \Td ))} T}{p} \right)},\\
&\| \nabla \rho \|_{L^2((0,T) \times \Td)} \leq \frac{\| \rho_0 \|_{L^2(\Td)}}{\sqrt{2 \nu}} \left( 1 + \| v \|_{L^{\infty}(0,T; W^{1,\infty}( \Td ))} T \exp{ \left(\| v \|_{L^{\infty}(0,T; W^{1,\infty}( \Td ))} T  \right) } \right)^{\frac{1}{2}},\\
&\|\partial_t\rho\|_{L^{2}(0,T; H^{-1}(\Td))}\le \nu \| \nabla \rho \|_{L^2((0,T) \times \Td)} + \| v \|_{L^{\infty}(0,T; L^{\infty}(\Td))} \| \rho \|_{L^{2} ((0,T)\times \Td )}.
    \end{align*}
\end{lem}

\begin{proof}

The first estimate just follows by integrating in space the equation, using the periodic boundary conditions. We now focus on the $L^p$ estimates. 
  \begin{align*}
        \frac{d}{dt} \int_\Td \rho^p & = p\nu \int_\Td \rho^{p-1} \Delta \rho - p \int_\Td \rho^{p-1} \dive(\rho v) = - p(p-1) \nu \int_\Td \rho^{p-2} |\nabla\rho|^2 - (p-1) \int_\Td \rho^p \dive (v) \\
        & \leq (p-1) \| v \|_{W^{1, \infty} ((0,T) \times \Td)} \int_\Td \rho^p.
    \end{align*}
    Hence, due to the Grönwall inequality
    \begin{equation*}
        \| \rho \|_{L^{\infty} (0,T; L^p(\Td))} \leq \|\rho_0 \|_{L^p(\Td)} \exp{\left( \frac{(p-1) \| v \|_{L^{\infty}(0,T; W^{1,\infty}( \Td ))} T}{p} \right)}.
    \end{equation*}
    Letting $p\rightarrow \infty$, we also obtain the $L^\infty$ estimate.
    Furthermore, from the first computation with $p=2$ it also follows that
    \begin{align*}
        2\nu \int_0^T \int_{\Td} |\nabla \rho |^2 \leq \int_{\Td} \rho_0^2 + \| v \|_{W^{1,\infty}( (0,T) \times \Td )} \int_0^T \int_{\Td} \rho^2 .
    \end{align*}
    Therefore,
    \begin{equation*}
        \sqrt{2 \nu} \| \nabla \rho \|_{L^2((0,T) \times \Td)} \leq \| \rho_0 \|_{L^2(\Td)} \left( 1 + \| v \|_{L^{\infty}(0,T; W^{1,\infty}( \Td ))} T \exp{ \left(\| v \|_{L^{\infty}(0,T; W^{1,\infty}( \Td ))} T  \right) } \right)^{\frac{1}{2}}.
    \end{equation*}
    Finally, for the time regularity, we take a test function $\varphi \in L^{2}(0,T; H^1(\Td))$ and
    \begin{align*}
        \left|\int_{0}^{T}\int_{\Td}\p_{t}\rho \varphi\diff x\diff t\right| & = \left| - \nu \int_0^T \int_{\Td} \nabla \varphi \cdot \nabla \rho + \int_0^T \int_{\Td} \nabla \varphi \cdot v \rho \right| \\
        & \leq \nu \| \nabla \rho \|_{L^2((0,T) \times \Td)} \| \nabla \varphi \|_{L^2((0,T)\times\Td)} \\
        & \quad + \| v \|_{L^{\infty}(0,T; L^{\infty}(\Td))} \| \rho \|_{L^{2}  ((0,T)\times \Td )} \| \nabla \varphi \|_{L^2((0,T)\times\Td)}.
    \end{align*}

Taking the supremum over $\varphi$ with $\|\varphi\|_{L^{2}(0,T; H^{1}(\T^d))}\le 1$ yields the $L^{2}(0,T; H^{-1}(\T^d))$ bound by duality.
\end{proof}

With these remarks, we use Schauder's fixed point theorem in order to prove the existence result.

\begin{proof}[Proof of Proposition~\ref{thm:Viscosity term Existence}]
Let us fix a time $T>0$. We split the proof into several steps. 

\textit{Step 0. Define the fixed point operator.} 
First we define 
\begin{equation*}
    \mathcal{T}: W^{1 ,\infty} ( [0,T] \times \Td ) \rightarrow L^{2}((0,T) \times \Td)
\end{equation*}
as the functional that maps a velocity $v$ with the solution of \eqref{eq:system_rewritten_boldv} for initial data $\rho_0$ and velocity $v$. From the previous computations, this map is well-defined. By a slight abuse of notation, as a second step of the construction of our fixed-point operator, we consider the map 
\begin{equation*}
    \bold v_{\ee,\alpha}: L^{2}((0,T) \times \Td) \rightarrow W^{1,\infty} ([0,T] \times \Td),
\end{equation*}
where we recall
\begin{equation*}
    \bold v_{\ee,\alpha} [\rho] = \nabla \left( - B_\ee [\rho \ast \widetilde\omega_\ee \ast \widetilde\omega_\ee] + \frac{m}{m-1} \widetilde\omega_{\ee} \ast (\rho \ast \widetilde\omega_\ee)^{m-1} - \ee^*  R_{\alpha}\ast\rho \right).
\end{equation*}
We realise that $\bold{v}_{\ee , \alpha} [\rho]$ is smooth since the kernels (i.e. $\omega_\ee$, $\widetilde{\omega}_\ee$, and $R_\alpha$) are smooth. Furthermore, we control the derivatives of the velocity,
\begin{equation}\label{estimate_boldv}
        \|D^{k} \bold v_{\ee,\alpha} [\rho]\|_{L^{\infty}(0,T; L^{\infty}(\Td))}\le C (\|\omega_1\|_{W^{k+1,\infty}}, \|R\|_{W^{k+1,\infty}}) (\|\rho \|_{L^{1}(\Td)}+ \| \rho \|_{L^{1}(\Td)}^{m-1})\le C,
\end{equation}    
where $C$ depends on $\ee,\alpha$ but these quantity are fixed. Finally, we define the operator we use for the fixed point argument as
\begin{equation*}
    S= \mathcal{T} \circ \bold{v}_{\ee, \alpha} : L^{2}((0,T) \times \Td) \rightarrow L^{2}((0,T) \times \Td).
\end{equation*}

\textit{Step 1. $S$ maps $X$ to $X$.} Taking into account Lemma~\ref{lem:est_nu} and estimate \eqref{estimate_boldv} we choose the constant   $A = \exp{((C(\|\omega_1\|_{W^{2,\infty}}, \|R\|_{W^{2,\infty}}) (\|\rho \|_{L^{1}(\Td)}+ \| \rho \|_{L^{1}(\Td)}^{m-1})T)/2)}$. We now show that the set
\begin{equation*}
    X\coloneqq \left\lbrace u \in L^{2}((0,T)\times \Td),\, \|u\|_{L^{2}(0,T;L^{2}(\Td))}\le A\sqrt{T} \|\rho_0 \|_{L^2(\Td)},\, u\ge 0 \text{ a.e. } \, \int_{\Td}u(t)=1\, \text{ for a.e. $t$}   \right\rbrace
\end{equation*}
is such that $S(X) \subseteq X$. The non-negativity of $\rho$ follows from the Maximum Principle considering $0$ as a subsolution. Hence, in view of Lemma~\ref{lem:est_nu} it follows that
\begin{equation*}
    \| S[\rho] \|_{L^{\infty} (0,T; L^2(\Td))} \leq \|\rho_0 \|_{L^2(\Td)} \exp{\left( \frac{\| \bold{v}_{\ee, \alpha} [\rho] \|_{L^{\infty}(0,T; W^{1,\infty}( \Td ))} T}{2} \right)} \leq A \| \rho_0 \|_{L^2 (\Td)}.
\end{equation*}
Therefore,
$$
 \| S[\rho] \|_{L^{2} (0,T; L^2(\Td))} \leq A\sqrt{T}\| \rho_0 \|_{L^2 (\Td)}
$$
and $S(X) \subseteq X$.

\textit{Step 2. $S$ is compact.} Using Lemma~\ref{lem:est_nu} one realises that 
\begin{equation*}
    \| S[\rho]\|_{L^2(0,T; H^1(\Td))} \leq C.
\end{equation*}
Furthermore, we also have the time regularity
\begin{equation*}
    \| \partial_t S[\rho] \|_{L^2(0,T; H^{-1}(\Td) )} \leq C.
\end{equation*}
%
%
Since $S[\rho]$ is bounded in the set 
$$
\{f \in L^{2}(0,T; H^{1}(\Td)), \, \p_{t}f \in L^{2}(0,T; H^{-1}(\Td))\}
$$
we can apply the Aubin-Lions Lemma to conclude compactness in $L^{2}((0,T)\times\Td)$. Therefore $S$ is compact in $X$.

\textit{Step 3. Continuity of $S$}. Let us consider a sequence $g_n \in X$ that converges to some $g \in L^2((0,T) \times \Td)$. Let us also consider the sequence $\rho_n = S[g_n]$. In \textit{Step 2} we prove compactness of $S$ in $X$. Hence, there exists a subsequence in $X$ such that $\rho_n$ converges to $\rho = S[g]$. Finally, since $\rho$ is the unique solution to the problem \eqref{eq:system_rewritten_boldv} with velocity $\bold{v}_{\ee, \alpha} [g]$ (uniqueness for such a parabolic equation is classical), the whole sequence converges. Thus $S$ is continuous.

\textit{Step 4. Schauder fixed point theorem}. Finally, we combine \textit{Step 1-3} with the Schauder fixed point theorem in order to conclude that there exists $\rho$, a weak solution of~\eqref{eq:Regularised Problem} on $(0, T)\times\Td$ with initial condition $\rho(0,\cdot)=\rho_{0}$. Since $\bold{v}_{\ee, \alpha} [\rho ]$ is smooth, from classical theory, the weak solution is also classical. 
\end{proof}

\subsection{Existence for fixed \texorpdfstring{$\ee, \alpha$}{epsilon, alpha}. Viscosity limit \texorpdfstring{$\nu \to 0$}{nu to 0}. Proof of \texorpdfstring{Theorem~\ref{thm:Nonlocal Existence}}{Theorem 1.6}}\label{sec:Viscosity limit}
For the sake of clarity, in the following subsection, when we refer to the solution of the problem \eqref{eq:Regularised Problem} presented in Proposition~\ref{thm:Viscosity term Existence} we will denote it by $\rho_{\nu}$.

    Let us take a sequence $\nu_k \rightarrow 0$ as $k \rightarrow \infty$. From Lemma~\ref{lem:est_nu} and Banach-Alaoglu Theorem, it follows that, up to a subsequence,
    \begin{equation*}
        \rho_{\nu_k}\overset{\ast}{\rightharpoonup}\rho \quad \text{weakly star in $L^{\infty}(0,T; L^{\infty}(\T^d))$}.
    \end{equation*}
   Moreover, the convolution by a kernel upgrades the weak convergence to a strong convergence in space, and using also the estimate on $\partial_t \rho_{\nu_k}$, which gives the same estimate on $\partial_t \rho_{\nu_k}\ast \omega$ for a smooth kernel $\omega$, we can apply the Aubin-Lions lemma to recover that 
    \begin{equation*}
        \bold{v}_{\ee, \alpha} [\rho_{\nu_k}] \to \bold{v}_{\ee, \alpha} [\rho] \quad \text{strongly in $L^{2}((0,T)\times\Td)$}.
    \end{equation*}
    Lemma~\ref{lem:est_nu} also shows that $\nu_k\nabla\rho_{\nu_k}$ converges to 0 strongly in $L^{2}((0,T)\times\Td)$. Indeed
    $$
    \|\nu_k\nabla\rho_{\nu_k}\|_{L^2}= \sqrt{\nu_k}\|\sqrt{\nu_k}\nabla\rho_{\nu_k}\|_{L^{2}}\le C\sqrt{\nu_k}.
    $$
    We know that $\rho_{\nu_k}$ is a classical solution of the problem \eqref{eq:system_rewritten_boldv}, so it is also a weak solution. Then, for all $\varphi \in C^{\infty}_c([0, T) \times \Td)$,
    \begin{equation*}
        -\int_0^T \int_\Td \rho_{\nu_k} \frac{\partial \varphi}{\partial t} \diff x \diff t - \int_{\Td} \rho_0 \varphi (0, x) \diff x =  \int_0^{T} \int_{\Td} \rho_{\nu_k} \bold{v}_{\ee, \alpha} [\rho_{\nu_k}] \cdot \nabla \varphi - \nu_k \int_0^T \int_\Td \nabla \rho_{\nu_k} \cdot \nabla \varphi 
    \end{equation*}
    converges to 
    \begin{equation*}
        -\int_0^T \rho \frac{\partial \varphi}{\partial t} \diff x \diff t - \int_{\Td} \rho_0 \varphi (0) \diff x =  \int_0^{T} \int_{\Td} \rho \bold{v}_{\ee, \alpha} [\rho] \cdot \nabla \varphi .
    \end{equation*}

\subsection{Uniqueness for fixed \texorpdfstring{$\ee, \alpha$}{epsilon, alpha}. Proof of \texorpdfstring{Theorem~\ref{thm:Nonlocal Uniqueness}}{Theorem 1.7}}\label{sec:Nonlocal Uniqueness}

    Consider $\rho_1$, $\rho_2$, two weak solutions of the problem \eqref{eq:Nonlocal CH} bounded in $L^{\infty}(0,T; L^{\infty}(\Td))$ with the same initial conditions. We take $\rho\coloneqq\rho_2-\rho_1$, its difference. In particular, $\rho(0, \cdot ) = 0$.  We introduce $\varphi$ such that
    $$
        -\Delta \varphi = \rho 
    $$
    which is well defined as $\int_{\T^d}\rho=0$.
    Observe that $\norm{\nabla\varphi (t , \cdot)}_{L^{2}(\Td)} = \norm{\rho(t, \cdot)}_{H^{-1}(\Td)}$. Also, 
    $$
    -\int_{\T^d}\rho\partial_t\varphi = -\frac{1}{2}\frac{d}{dt}\|\nabla\varphi\|_{L^{2}(\Td)}^2
    $$
    Hence, using $\varphi$ as a test function, we recover that
    \begin{align*}
         -\int_\Td \rho \frac{\partial \varphi}{\partial t} & = - \int_\Td \rho \nabla B_\ee [\rho_2 \ast \widetilde\omega_\ee \ast \widetilde\omega_\ee] \cdot \nabla \varphi - \int_\Td \rho_1 \nabla B_\ee [\rho \ast \widetilde\omega_\ee \ast \widetilde\omega_\ee] \cdot \nabla \varphi \\
         & \quad + \frac{m}{m-1} \int_\Td \rho \nabla \widetilde\omega_\ee \ast (\rho_2 \ast \widetilde\omega_\ee)^{m-1} \cdot \nabla \varphi \\
         & \quad + \frac{m}{m-1} \int_\Td \rho_1 \nabla \widetilde\omega_\ee \ast \left[ (\rho_2 \ast \widetilde\omega_\ee)^{m-1} - (\rho_1 \ast \widetilde\omega_\ee)^{m-1} \right] \cdot \nabla \varphi \\
         & \quad - \ee^* \int_\Td \rho \nabla ( R_\alpha \ast \rho_2 ) \cdot \nabla \varphi - \ee^* \int_\Td \rho_1 \nabla ( R_\alpha \ast \rho ) \cdot \nabla \varphi \\
         & = I_1 + I_2 + I_3 + I_4 + I_5 + I_6.
    \end{align*}
    Concerning $I_1$ we use $\rho = - \Delta \varphi$ and the identity
    \begin{equation}\label{eq:Trick}
        \nabla u \Delta u = \dive(\nabla u \otimes \nabla u ) - \frac{1}{2} \nabla | \nabla u |^2 .
    \end{equation}
    Thereby, we obtain that
    \begin{align*}
        |I_1|  & = \left| \int_\Td (\nabla \varphi \otimes \nabla \varphi) : D^2 B_\ee [\rho_2 \ast \widetilde\omega_\ee \ast \widetilde\omega_\ee] - \frac{1}{2} \int_\Td |\nabla \varphi |^2 \Delta B_\ee [ \rho_2 \ast \widetilde\omega_\ee \ast \widetilde\omega_\ee] \right| \\
        & \leq C \| \nabla \varphi \|_{L^2(\Td)}^2 \| D^2 B_\ee[ \rho_2 \ast \widetilde\omega_\ee \ast \widetilde\omega_\ee] \|_{L^\infty (\Td)} \\
        & \leq C \| \nabla \varphi \|_{L^2(\Td)}^2 \frac{1}{\ee^2} \left( \|D^2 \widetilde\omega_\ee \ast \widetilde\omega_\ee \|_{L^\infty (\Td)} + \| D^2 \omega_\ee \ast \widetilde\omega_\ee \ast \widetilde\omega_\ee \|_{L^\infty (\Td)} \right) \| \rho_2 \|_{L^1 (\Td)} ,
    \end{align*}
    which is uniformly bounded by $C\|\nabla\varphi\|_{L^{2}}^2$ since the mollifiers have two bounded derivatives. In order to deal with $I_2$ we use again $\rho = - \Delta \varphi$. By properties of the convolution, we have that
    \begin{align*}
        |I_2| & = \left| \int_\Td \rho_1 \nabla B_\ee [\rho \ast \widetilde\omega_\ee \ast \widetilde\omega_\ee] \cdot \nabla \varphi \right| \\
        & \leq \left| \frac{1}{\ee^2} \int_\Td \rho_1 \nabla (\rho \ast \widetilde\omega_\ee \ast \widetilde\omega_\ee ) \cdot \nabla \varphi \right| +  \left| \frac{1}{\ee^2} \int_\Td \rho_1 \nabla (\rho \ast \omega_\ee \ast \widetilde\omega_\ee \ast \widetilde\omega_\ee ) \cdot \nabla \varphi \right| \\
        & = \left| \frac{1}{\ee^2} \int_\Td \rho_1  (\Delta \varphi \ast \nabla ( \widetilde\omega_\ee \ast \widetilde\omega_\ee )) \cdot \nabla \varphi \right| +  \left| \frac{1}{\ee^2} \int_\Td \rho_1  (\Delta \varphi \ast \nabla (\omega_\ee \ast \widetilde\omega_\ee \ast \widetilde\omega_\ee) ) \cdot \nabla \varphi \right| \\
        & = \left| \frac{1}{\ee^2} \sum_{i,j} \int_\Td \rho_1 (\partial_j \varphi \ast \partial_{ij} (\widetilde\omega_\ee \ast \widetilde\omega_\ee)) \partial_i \varphi  \right| + \left| \frac{1}{\ee^2} \sum_{i,j} \int_\Td \rho_1 (\partial_j \varphi \ast \partial_{ij} (\omega_\ee \ast \widetilde\omega_\ee \ast \widetilde\omega_\ee)) \partial_i \varphi  \right| \\
        & \leq \frac{1}{\ee^2} C \| \rho_1 \|_{L^\infty(\Td)} \left( \|D^2 \widetilde\omega_\ee \ast \widetilde\omega_\ee \|_{L^\infty (\Td)} + \| D^2 \omega_\ee \ast \widetilde\omega_\ee \ast \widetilde\omega_\ee \|_{L^\infty (\Td)} \right) \| \nabla \varphi \|_{L^2(\Td)}^2.
    \end{align*}
    We continue with $I_3$. We use $\rho = - \Delta \varphi$ and \eqref{eq:Trick} again and it follows that
    \begin{align*}
        |I_3 | & = \left| \frac{m}{m-1} \int_\Td \rho \nabla \widetilde\omega_\ee \ast (\widetilde\omega_\ee \ast \rho_2)^{m-1} \cdot \nabla \varphi \right| \\
        & = \left|  \frac{m}{m-1} \int_\Td (\nabla \varphi \otimes \nabla \varphi ) : [D^2 \widetilde\omega_\ee \ast (\widetilde\omega_\ee \ast \rho_2)^{m-1}] - \frac{m}{2(m-1)} \int_\Td | \nabla \varphi |^2 [\Delta \widetilde\omega_\ee \ast (\widetilde\omega_\ee \ast \rho_2)^{m-1}] \right| \\
        & \leq C \| \nabla \varphi \|_{L^2 (\Td)}^2 \| D^2 \widetilde\omega_\ee \|_{L^\infty (\Td)} \| (\widetilde\omega_\ee \ast \rho_2 )^{m-1} \|_{L^1(\Td)}\\
        &\le C\| \nabla \varphi \|_{L^2 (\Td)}^2 .
    \end{align*}
    For $I_4$, we recall $\rho_i \in L^{\infty}(0,T; L^{\infty}(\Td))$. Then,  $\rho_i \ast \widetilde\omega_\ee$ is uniformly bounded by a constant $A$. For $m \geq 2$, the function $f(s)=s^{m-1}$ is Lipschitz in $[0,A]$. Combining everything, for $m\geq 2$, we get that
    \begin{align*}
        |I_4| & = \frac{m}{m-1} \left| \int_\Td \rho_1 \nabla \widetilde\omega_\ee \ast \left[ (\rho_2 \ast \widetilde\omega_\ee)^{m-1} - (\rho_1 \ast \widetilde\omega_\ee)^{m-1} \right] \cdot \nabla \varphi \right| \\
        & \leq C \int_\Td |\rho_1| |\nabla \widetilde\omega_\ee| |\rho_2 \ast \widetilde\omega_\ee - \rho_1 \ast \widetilde\omega_\ee | |\nabla\varphi|\\
        & \leq C \| \rho_1 \|_{L^\infty(\Td)} \| \nabla \widetilde\omega_\ee \|_{L^\infty(\Td)} \| \rho\ast\widetilde\omega_\ee \|_{L^2(\Td)} \|\nabla \varphi \|_{L^2(\Td)} \\
        & = C \| \rho_1 \|_{L^\infty(\Td)} \| \nabla \widetilde\omega_\ee \|_{L^\infty(\Td)} \left\| \sum_i\partial_i\varphi\ast\partial_i\widetilde\omega_\ee \right\|_{L^2(\Td)} \|\nabla \varphi \|_{L^2(\Td)} \\
        & \leq C \| \nabla \varphi \|_{L^2(\Td)}^2.
    \end{align*}
    The terms $I_5$ and $I_6$ are analogous to $I_3$ and $I_4$ respectively for $m=2$. In the end, we recover 
    $$
    \frac{d}{dt}\|\nabla\varphi(t)\|_{L^{2}(\Td)}^2\le C\|\nabla\varphi(t)\|_{L^{2}(\Td)}^2, \quad \|\nabla \varphi (0, \cdot )\|_{L^2(\Td)} = 0.
    $$

Gr\"onwall's inequality yields $\| \nabla \varphi (t, \cdot )\|_{L^2(\Td)} = \| \rho (t , \cdot ) \|_{H^{-1}(\Td)} = 0$ for all $0 < t < T$. Note that this reasoning also yields continuous dependence with respect to the initial data.

\section{Convergence nonlocal to local}\label{sec:Non-L to L}
In the following, to make the distinction, we use the notation $\rho_{\ee, \alpha}$ for the solutions of the problem \eqref{eq:Nonlocal CH}. We keep the notation $\rho_\nu$ for the solutions of \eqref{eq:Regularised Problem} since whenever we require them, the parameters $\ee$ and $\alpha$ are fixed. The main goal of this section is to prove convergence of the solutions of the problem \eqref{eq:Nonlocal CH} to the solutions of \eqref{eq:CH} when $m=2$. Therefore, from now on we assume $m=2$. In order to achieve this goal, we divide this section into three subsections. First, in Subsection~\ref{sec:NL GF properties} we introduce some properties and \textit{a priori} bounds that follow from the gradient flow structure of the problem. Afterwards, in Subsection~\ref{sec:Uniform ee} we compile preliminary results regarding compactness in $\ee$ and $\alpha$. Finally, in Subsection~\ref{sec:NL to L} we show the nonlocal to local convergence, and we prove Theorem~\ref{thm:Nonlocal to local}.

\subsection{Gradient flow structure of the problem}\label{sec:NL GF properties}

Lemma~\ref{lem:est_nu} provides estimates for the viscous equation with a prescribed velocity. Applied to the viscous approximation of~\eqref{eq:Nonlocal CH} these bounds are not uniform in $\ee, \alpha$ because they involve $\|\bold v_{\ee,\alpha}[\rho]\|_{W^{1,\infty}(\Td)}$. Therefore, the aim of this subsection is to study the gradient flow structure of the problem in order to recover estimates uniformly in $\ee, \alpha$. In particular, we take advantage of the following \textit{a priori} estimates.

\begin{prop}\label{lem:Gradient flow}
    Solutions of the problem \eqref{eq:Nonlocal CH} constructed in Theorem~\ref{thm:Nonlocal Existence} are such that for $t\in(0,T)$ 
    \begin{equation}\label{eq:Grad flow inequality}
        \begin{split}
         &\mf_{\ee,\alpha} [\rho_{\ee, \alpha}(t, \cdot )] + \int_{0}^{t} \int_\Td \rho_{\ee, \alpha} \left|  \bold{v}_{\ee, \alpha}[\rho_{\ee, \alpha}] \right|^2 \leq \mf_{\ee,\alpha} [\rho_{\ee, \alpha}(0, \cdot )] ,\\
          &\|\partial_t\rho_{\eps,\alpha}\|_{L^{2}(0,t; H^{-s}(\Td))}\le C \|\rho_{0}\|_{L^{1}}^{\frac{1}{2}}\|\sqrt{\rho_{\eps,\alpha}} \bold{v}_{\ee,\alpha}\|_{L^{2}((0,t)\times\T^d)}
        \end{split}
    \end{equation}
 where $s>\frac{d}{2}+1$. In particular, from Proposition~\ref{prop:Some properties} and since the initial condition is of bounded free energy, there exists $C$  independent of $\ee$ and $\alpha$ such that
 \begin{align}
      \sup_{t\in (0,T)}\| \nabla (\rho_{\ee, \alpha} \ast \widetilde\omega_\ee ) (t, \cdot) \|_{L^2 (\Td)} & \leq C, \label{eq:Uniform H1 Section 4} \\
      \sup_{t\in (0,T)} \| ( \rho_{\ee, \alpha} \ast \widetilde\omega_\ee )(t, \cdot)  \|_{L^p (\Td)} & \leq C \label{eq:Uniform Lp Section 4}\\
      \sup_{t\in (0,T)}
        \|S_{\eps}[\rho\ast\widetilde{\omega}_{\eps}]\|_{L^2(\T^d\times\T^d)}&\le C \label{eq:Uniform H1_2 Section 4}.
 \end{align}
 for any $p\in[1,2^*]$, $d\neq 2$, and for any $p\in[1,2^*)$ when $d=2$ and where $S_\ee$ corresponds with the non-local operator \eqref{def:S}.
\end{prop}

\begin{proof}
    We divide the proof into several steps.

    \textit{Step 1. Dissipation of the free energy.} 
    Let $t\in(0,T)$. Let us take the problem \eqref{eq:Regularised Problem} and consider its corresponding free-energy
    \begin{equation*}
        \mathcal{F}_{\ee , \alpha , \nu} [\rho]  = \nu \int_\Td \rho \log \rho + \mathcal{F}_{\ee , \alpha} [\rho].
    \end{equation*}
    Since the solution $\rho_\nu$ of \eqref{eq:Regularised Problem} is classical, we can take its time derivative and its gradient flow formulation \cite{Carrillo_McCann_Villani03} to obtain that
    \begin{align*}
        \frac{\diff}{\diff t} \mathcal{F}_{\ee, \alpha, \nu} [\rho_\nu] & =  \int_\Td \partial_t \rho_\nu \frac{\delta \mathcal{F}_{\ee, \alpha, \nu}}{\delta \rho} [\rho_\nu] = \int_\Td \dive \left( \rho_\nu \nabla \frac{\delta \mathcal{F}_{\ee, \alpha, \nu}}{\delta \rho} [\rho_\nu] \right) \frac{\delta \mathcal{F}_{\ee, \alpha, \nu}}{\delta \rho} [\rho_\nu] \\ & = - \int_\Td \rho_\nu \left| \nabla \frac{\delta \mathcal{F}_{\ee, \alpha, \nu}}{\delta \rho} [\rho_\nu] \right|^2 
        = - \int_\Td \rho_\nu \left| \nu \nabla \log \rho_\nu - \bold{v}_{\ee, \alpha}[\rho_\nu] \right|^2.
    \end{align*}
    From here, it follows 
    \begin{equation}\label{eq:Gradient flow structure}
        \int_{0}^{t} \int_\Td \rho_\nu \left| \nu \nabla \log \rho_\nu - \bold{v}_{\ee, \alpha}[\rho_\nu] \right|^2 \leq \mf_{\ee,\alpha,\nu} [\rho_\nu(0, \cdot )] - \mf_{\ee,\alpha,\nu} [\rho_\nu(t, \cdot )] .
    \end{equation}
    Let us take a sequence $\nu_k \rightarrow 0$ and $\rho_{\nu_k}$ a sequence of solutions of the problem \eqref{eq:Regularised Problem}. From \eqref{eq:Gradient flow structure}, Fatou's Lemma, and the previous convergences, the result follows.
   
    \textit{Step 2. The time derivative.} 
    Concerning the estimate on the time derivative, we write (in the weak sense) 
    $$
    \partial_t \rho_{\ee,\alpha} = \dive\left(\sqrt{\rho_{\ee,\alpha}}\sqrt{\rho_{\ee,\alpha}}\bold{v}_{\ee,\alpha}\right).
    $$
    By conservation of mass we get $\|\sqrt{\rho_{\eps,\alpha}}\|_{L^{\infty}(0,t; L^{2}(\Td))}\le\|\rho_0\|_{L^{1}}^{\frac 12}$. Moreover $\sqrt{\rho_{\ee,\alpha}}\bold{v}_{\ee,\alpha}$ is bounded in $L^{2}((0,t)\times\Td)$ from the first estimate in \eqref{eq:Grad flow inequality}. Therefore the product is bounded in $L^{2}(0,t; L^{1}(\Td))$. Taking a test function and working by duality then yields the result. 
\end{proof}

\subsection{Uniform estimates in \texorpdfstring{$\ee, \alpha$}{epsilon, alpha}}\label{sec:Uniform ee}
We discuss other several estimates concerning $\rho_{\ee, \alpha}$ uniform in $\ee$ and $\alpha$. Until now we only used the free energy of the system and its gradient flow structure. There exists another Lyapunov functional, typical for fourth order equation like the thin-film/ Cahn-Hilliard equations, which is called the entropy. We recall its definition 
\begin{equation*}
    \Phi [\rho] = \int_\Td \rho (\log \rho - 1)  \dx .
\end{equation*}
In this subsection and until the end we assume that $\widetilde\ee$ and $\ee^*$ satisfy~\eqref{def:widetilde ee}. In particular, $\frac{\ee^2}{\widetilde\ee^{d+6}}\to 0$ and $\frac{\widetilde{\ee}}{\ee^*}\to 0$  as $\ee\to 0$.
We obtain the following estimates, which can be made rigorous in the approximating scheme and then by sending the viscosity $\nu$ to zero. 

\begin{prop}\label{prop:More a priori bounds}
    Assume $\rho_{\ee, \alpha}$ is a solution of \eqref{eq:Nonlocal CH} constructed in Theorem~\ref{thm:Nonlocal Existence}. Assume furthermore that $\sup_{0<\ee,\alpha<1}\mathcal{F}_{\ee,\alpha} [\rho_0], \Phi [\rho_0] < \infty$. Then, there exists $C$ independent of $\ee$ such that:
    \begin{align}
        \frac{1}{2} \int_{0}^{T}\iint_{\T^d\times\T^d}  \left| S_\ee [ \nabla (\rho_{\ee, \alpha} \ast \widetilde\omega_\ee)] \right|^2  
        & \leq  C  , \label{eq:Bounds S_ee} \\
         \frac{1}{2} \int_{0}^{T}\iint_{\T^d\times\T^d}  \left| S_\ee [ (\rho_{\ee, \alpha} \ast \widetilde\omega_\ee)] \right|^2  
        & \leq  C  , \label{eq:Bounds S_ee2} \\
        \ee^* \int_{0}^{T}\int_\Td |\nabla (\rho_{\ee, \alpha} \ast  R^{\f 12}_\alpha) |^2  & \le  C , \label{eq:Bounds H1} \\
        \sup_{t\in[0,T]}\int_{\Td} \rho_{\ee, \alpha} (t, \cdot) \log \rho_{\ee, \alpha}(t, \cdot)  & \leq C, \label{eq:Linfty LlogL} 
    \end{align}
    where $S_\ee$ corresponds with the non-local operator \eqref{def:S}.
\end{prop}

\begin{proof}[Proof of Proposition~\ref{prop:More a priori bounds}]
The following computations are formal but can be made rigorous by coming back to the viscosity system where $\nu>0$ and then passing to the limit $\nu\to 0$. For simplicity, we omit this technicality here. Let us consider $\rho_{\ee, \alpha}$ a solution of \eqref{eq:Nonlocal CH} and the non-local operators \eqref{def:Bee}, \eqref{def:S}. Therefore, we compute in order to obtain that
    \begin{align*}
        \frac{\diff}{\diff t} \Phi [\rho_{\ee, \alpha}] & = \int_\Td \log \rho_{\ee, \alpha} \frac{\partial \rho_{\ee, \alpha}}{\partial t} = - \int_\Td \log \rho_{\ee, \alpha} \dive (\rho_{\ee, \alpha} \bold v_{\ee,\alpha} [\rho_{\ee, \alpha}]) \\
        & =  \int_\Td \nabla \rho_{\ee, \alpha} \cdot \bold v_{\ee,\alpha} [\rho_{\ee, \alpha}] \\
        & = - \iint_{\T^d\times\T^d} |S_\ee [\nabla (\rho_{\ee, \alpha} \ast \widetilde\omega_\ee)]|^2 - 2 \int_\Td (\rho_{\ee, \alpha} \ast \widetilde\omega_\ee)  \Delta (\rho_{\ee, \alpha} \ast \widetilde\omega_\ee) \\
        & \quad - \ee^*\int_\Td | \nabla ( \rho_{\ee, \alpha} \ast  R^{\f 12}_{\alpha} ) |^2.
    \end{align*}
    Therefore, if we take the integral in time it follows that
    \begin{align*}
        \Phi [\rho_{\ee, \alpha} (t, \cdot)] +& \int_{0}^{t}\iint_{\T^d\times\T^d}  \left| S_\ee [ \nabla (\rho_{\ee, \alpha} \ast \widetilde\omega_\ee)] \right|^2  + \ee^*\int_{0}^{t} \int_\Td | \nabla ( \rho_{\ee, \alpha} \ast  R^{\f 12}_{\alpha} ) |^2 \\
        &  = \Phi [\rho_0]  - 2 \int_0^t\int_\Td (\rho_{\ee, \alpha} \ast \widetilde\omega_\ee)  \Delta (\rho_{\ee, \alpha} \ast \widetilde\omega_\ee) .
    \end{align*}
    We can apply Young's inequality to the second term on the RHS and we recover
    \begin{align*}
         \Phi [\rho_{\ee, \alpha} (t, \cdot)] +& \int_{0}^{t}\iint_{\T^d\times\T^d}  \left| S_\ee [ \nabla (\rho_{\ee, \alpha} \ast \widetilde\omega_\ee)] \right|^2  + \ee^*\int_{0}^{t} \int_\Td | \nabla ( \rho_{\ee, \alpha} \ast  R^{\f 12}_{\alpha} ) |^2 \\
        &  \leq \Phi [\rho_0]  + \frac{1}{2 C } \int_{0}^{t}\int_\Td | \Delta (\rho_{\ee, \alpha} \ast \widetilde\omega_\ee) |^2 + C  \int_{0}^{t}\int_\Td (\rho_{\ee, \alpha} \ast \widetilde\omega_\ee)^{2} .
    \end{align*}
    Finally, we apply the non-local inequality \eqref{eq:L2 NL Laplacian} to the second term on the RHS. In the previous step we could choose $C$ due to Young's inequality. We choose it so that it is the same as in \eqref{eq:L2 NL Laplacian}. Hence, we get
    \begin{align*}
         \Phi [\rho_{\ee, \alpha} (t, \cdot)] +& \int_{0}^{t}\iint_{\T^d\times\T^d}  \left| S_\ee [ \nabla (\rho_{\ee, \alpha} \ast \widetilde\omega_\ee)] \right|^2  + \ee^*\int_{0}^{t} \int_\Td | \nabla ( \rho_{\ee, \alpha} \ast  R^{\f 12}_{\alpha} ) |^2  \\
        &  \leq \Phi [\rho_0]  + \frac{1}{2} \int_{0}^{t}\iint_{\T^d\times\T^d} \left| S_\ee [ \nabla (\rho_{\ee, \alpha} \ast \widetilde\omega_\ee)] \right|^2 + C \left( \int_{0}^{t}\int_\Td (\rho_{\ee, \alpha} \ast \widetilde\omega_\ee)^{2} + \frac{\ee^2}{\widetilde\ee^{d+6}} \right).
    \end{align*}
Moreover, from  Proposition~\ref{lem:Gradient flow} we can bound 
$$
\int_{0}^{T}\int_\Td (\rho_{\ee, \alpha} \ast \widetilde\omega_\ee)^{2}\le C.
$$
From this inequality and the initial boundedness of the entropy, the result follows. 
\end{proof}

Thanks to these estimates, we recover various compactness results that we state here.

\begin{prop}[Weak and Strong Convergences]\label{lem:Frechet-Kolmogorov}
Let $\{\rho_{\varepsilon,\alpha}\}_{\varepsilon,\alpha}$ be as above. Then, up to the extraction of (non–relabeled) subsequences, the following convergences hold:
\begin{itemize}
    \item For each fixed $\varepsilon>0$, there exists $\rho_{\ee,0}\in L^{\infty}(0,T;L^{1}(\Td))\cap L^{2}(0,T; H^{1}(\Td))$ such that as $\alpha\to 0$:
    \begin{equation*}
    \begin{split}
        &\rho_{\varepsilon,\alpha}\ \rightharpoonup\ \rho_{\varepsilon,0}
        \quad\text{weakly in }L^{1}((0,T)\times \Td),\\
        & \rho_{\varepsilon,\alpha}\ast R^{\f 12}_{\alpha}\rightharpoonup\rho_{\varepsilon,0}
        \quad\text{weakly in }L^{2}(0,T;H^{1}(\Td)),\\
        &\rho_{\varepsilon,\alpha}\ast R^{\f 12}_{\alpha}\to\rho_{\varepsilon,0}
        \quad\text{strongly in }L^{2}((0,T)\times\Td).\\
    \end{split}
    \end{equation*}
    Furthermore, $\rho_{\ee, 0}$ is a weak solution of~\eqref{eq:Nonlocal CH} in the sense of Definition~\ref{def:Weak solution} and for all $t\in(0,T)$ it satisfies the bounds 
    \begin{equation}\label{eq:H1 bound}
        \ee^* \| \nabla \rho_{\ee , 0}  \|_{L^2 ((0,T)\times \Td)}^2 \le C, \quad \|\partial_t\rho_{\eps,0}\|_{L^{2}(0,T; H^{-s}(\Td))}\le C,\quad  \int_{\Td}\rho_{\ee,0}(t, \cdot)|\log\rho_{\ee,0}(t, \cdot)|\le C,
    \end{equation}
    where $C$ is independent of $\ee$.

    \item There exists $\rho\in L^{\infty}(0,T; L^{1}(\Td))\cap  L^{2}(0,T; H^{1}(\Td))$ such that, (up to a subsequence not relabeled) as $\ee\to 0$: 
    \begin{align}
        &\rho_{\varepsilon,0}\rightharpoonup\rho
        \quad\text{weakly in }L^{1}((0,T)\times \Td),\notag \\
        &\rho_{\varepsilon,0}\ast\widetilde\omega_{\varepsilon}\to\rho\quad\text{strongly in }L^{2}(0,T ; H^1 (\Td) ). \label{eq:strong H1 convolution}
    \end{align}
\end{itemize}

\end{prop}

\begin{proof}
Note that the first part of this proposition is only here to mention that we can construct solutions of the system with $\alpha=0$, as later we improve on these results by showing that there is even a rate of convergence as $\alpha\to 0$ with a different method. Therefore, for clarity, we only mention the key ideas. We recall that $\widetilde\ee$ and $\ee^*$ satisfy~\eqref{def:widetilde ee}. The convergence as $\alpha\to 0$ follows from Proposition~\ref{lem:Gradient flow} and Proposition~\ref{prop:More a priori bounds}:
\begin{itemize}
    \item The weak convergence in $L^{1}((0,T)\times \Td)$ follows from the uniform bound on $\int_{\Td}\rho_{\ee,\alpha}\log \rho_{\ee,\alpha}$ by~\eqref{eq:Linfty LlogL} and the Dunford-Pettis theorem.  In fact it even yields weak star convergence in $L^{\infty}(0,T; L^{1}(\T^d))$.
    \item The weak convergence in $L^{2}(0,T;H^{1}(\T^d))$ follows from~\eqref{eq:Bounds H1} as well as the nonlocal Poincaré inequality, see for instance~\cite[Lemma C.3]{Elbar_Skrzeczkowski23}.
    \item Using~\eqref{eq:Bounds H1} as well as the estimate on the time derivative in $L^2 (0,T;H^{-s} (\Td))$ for $s > \frac{d}{2} + 1$ from Proposition~\ref{lem:Gradient flow} we deduce the strong convergence by the Aubin-Lions Lemma with the spaces $H^1(\Td) \subset L^2 (\Td ) \subset H^{-s} (\Td)$.
\end{itemize}

Furthermore, $\rho_{\ee,0}$ can then be identified as a weak solution of~\eqref{eq:Nonlocal CH} with $\alpha=0$, as the previous convergences are enough to pass to the limit in the equation for $\rho_{\eps,\alpha}$. Let us note that the weak convergence on $\rho_{\eps,\alpha}$ is improved to a strong one when considering some terms of type $\omega_{\eps}$ for smooth kernels $\omega_{\eps}$. In fact, the most difficult term to treat is $\ee^*\dive(\rho_{\eps,\alpha}\nabla R_\alpha\ast\rho_{\ee,\alpha})$. However, this term can be treated exactly as in the paper of Lions and Mas-Gallic~\cite[Theorem 2]{Lions_Mas-Gallic01}, which considers the DPA for the porous medium equation; see also~\eqref{eq:L-MG}, where we detail a similar term as $\ee\to 0$. In the limit, this term yields
$$
    \ee^*\int \rho_{\eps,0}\nabla \rho_{\eps,0}\cdot\nabla\varphi=-\frac{\ee^*}{2}\int \rho_{\eps,0}^2\Delta\varphi
$$
in the weak formulation against a test function $\varphi\in C_{c}^{\infty}([0,T)\times \Td)$. The estimates on $\rho_{\eps,0}$ follow by lower semi-continuity of the norms and Fatou's Lemma. 

The convergence with respect to $\ee$ follows the same lines.  The most technical point is the last strong convergence of the gradient. But this last convergence can be adapted from \cite[Appendix D]{Elbar_Skrzeczkowski23}, using the nonlocal estimate of the gradient in Proposition
~\ref{prop:More a priori bounds} and the estimate on $\partial_t\rho_{\eps,0}$ presented in Proposition~\ref{lem:Gradient flow}.  
\end{proof}

\subsection{Nonlocal to local convergence. Proof of \texorpdfstring{Proposition~\ref{thm:Nonlocal to local H1}}{Proposition 1.8}}\label{sec:NL to L}

With these auxiliary results from the previous subsection, we are now able to show convergence of the solutions of \eqref{eq:Nonlocal CH} to the solutions of \eqref{eq:CH}.

\begin{proof}[Proof of Proposition~\ref{thm:Nonlocal to local H1}]
    In Proposition~\ref{prop:More a priori bounds} we have proved the first part of the proposition, that is, the convergence as $\alpha\to 0$ as well as the compactness with respect to $\ee$. It remains to show that the limit $\rho$ is a solution of~\eqref{eq:CH} in the sense of Definition~\ref{def:Weak solution local}. 

    \textit{Step 1. Nonlocal equation.} Let us consider a test function $\varphi \in C_c^\infty ([0,T) \times \Td)$ in the weak formulation of $\rho_{\ee,0}$. The main terms to treat are:
    \begin{align*}
        &I\coloneqq\int_0^T \! \! \int_\Td \! \dive \left( \rho_{\ee, 0} \nabla \left( B_\ee [\rho_{\ee, 0} \ast \widetilde\omega_\ee \ast \widetilde\omega_\ee]  \right) \right) \varphi , \\
        &J\coloneqq- 2 \int_0^T \! \! \int_\Td \! \dive \left(  \rho_{\ee, 0} \nabla \left( \widetilde{\omega}_\ee \ast \widetilde\omega_\ee \ast \rho_{\ee, 0} \right) \right) \varphi , \\
        & \mathfrak{R}\coloneqq- \ee^*\int_0^T \int_\Td \rho_{\ee, 0}\nabla\rho_{\ee,0}\cdot\nabla\varphi.
    \end{align*}
    We now proceed to study the convergence of the diffusion and aggregation terms separately ($I$ and $J$). Finally, we will deal with the remainder term $\mathfrak{R}$. For the ease of presentation in the remainder of the proof, we will use $\rho_\ee$ for $\rho_{\ee, 0}$.

    \textit{Step 2. Convergence of the diffusion term $I$.} In order to deal with the diffusive term, we take advantage of the properties of the nonlocal operator $S_\ee$ introduced in Lemma~\ref{lem:S properties}. We start by integrating by parts
    \begin{align*}
        I  = \int_0^T \int \dive \left( \rho_{\ee} \nabla B_\ee [\rho_{\ee} \ast \widetilde{\omega}_\ee \ast \widetilde\omega_\ee] \right) \varphi \dx \dt  = - \int_0^T \int \widetilde\omega_\ee \ast (\rho_{\ee} \nabla \varphi ) \cdot \nabla B_\ee [\rho_{\ee} \ast \widetilde{\omega}_\ee ] \dx \dt,
    \end{align*}
    where we also used the symmetry of the kernel to apply the formula 
    $$
    \int_{\T^d} (f\ast \widetilde{\omega}_\ee)g = \int_{\T^d} f (g\ast\widetilde{\omega}_\ee).
    $$
    Moreover, we use symmetry again 
    \begin{align*}
        I & = - \frac{1}{2} \int_0^T \iint \frac{\omega_\ee (y)}{\ee^2}  \left( (\widetilde{\omega}_\ee \ast (\rho_{\ee} \nabla \varphi) ) (x) - (\widetilde{\omega}_\ee \ast (\rho_{\ee} \nabla \varphi) ) (x-y) \right) \\
        & \hspace{41mm} \cdot \left( \nabla (\rho_{\ee} \ast \widetilde\omega_\ee)(x) - \nabla (\rho_{\ee} \ast \widetilde\omega_\ee)(x-y) \right) \dy \dx \dt \\
        & = \frac{1}{2} (I_1 + I_2),
    \end{align*}
    where
    \begin{equation*}
        I_1 = - \int_{0}^{T}\iint \frac{\omega_{\ee}(y)}{\ee^2}\big[((\omegat\ast\rho_{\ee})\nabla\varphi)(x)-((\omegat\ast\rho_{\ee})\nabla\varphi)(x-y)\big] \cdot \big[\nabla(\rho_{\ee}\ast \omegat)(x)-\nabla(\rho_{\ee}\ast \omegat)(x-y) \big] \dy \dx \dt
    \end{equation*}
    and
    \begin{align*}
        I_2 & = - \int_{0}^{T}\iint \frac{\omega_{\ee}(y)}{\ee^2} \big[ (\omegat\ast(\rho_{\ee}\nabla\varphi))(x) - ((\omegat\ast\rho_{\ee})\nabla\varphi)(x) \\
        & \hspace{34mm} - (\omegat\ast(\rho_{\ee}\nabla\varphi))(x-y) + ((\omegat\ast\rho_{\ee})\nabla\varphi)(x-y) \big]\\
        & \hspace{30mm} \cdot \big[ \nabla(\rho_{\ee}\ast\omegat)(x)-\nabla(\rho_{\ee}\ast\omegat)(x-y) \big] \dy\dx \dt .
    \end{align*}

    \textit{Step 2a. Convergence of $I_1$.} 
    Let us start dealing with $I_1$. Due to  integration by parts, 
    it follows that
    \begin{align*}
         I_1 & = - 2 \int_0^T \iint S_\ee [ \nabla (\rho_\ee \ast \widetilde\omega_\ee)] S_\ee [(\rho_\ee \ast \widetilde\omega_\ee) \nabla \varphi] \dy \dx \dt \\
        & = 2 \int_0^T \iint S_\ee [\rho_\ee \ast \widetilde\omega_\ee] S_\ee [\nabla (\rho_\ee \ast \widetilde\omega_\ee) \cdot \nabla \varphi] \dy \dx \dt + 2 \int_0^T \iint S_\ee [\rho_\ee \ast \widetilde\omega_\ee] S_\ee [ (\rho_\ee \ast \widetilde\omega_\ee ) \Delta \varphi ] \dy \dx \dt \\
        & = I_{11} + I_{12}.
    \end{align*}
    We can split $I_{11}$ into three terms to make it easier to study. We have that
    \begin{align*}
        I_{11} & = 2 \int_0^T \iint S_\ee [\rho_\ee \ast \widetilde\omega_\ee] S_\ee [\nabla (\rho_\ee \ast \widetilde\omega_\ee)] \cdot \nabla \varphi \dy \dx \dt + 2 \int_0^T \iint S_\ee [\rho_\ee \ast \widetilde\omega_\ee] \nabla (\rho_\ee \ast \widetilde\omega_\ee) \cdot S_\ee [\nabla \varphi ] \dy \dx \dt \\
        & \quad + 2 \int_0^T \iint S_\ee [\rho_\ee \ast \widetilde\omega_\ee] \left( S_\ee [\nabla(\rho_\ee \ast \widetilde\omega_\ee) \cdot \nabla \varphi] - S_\ee [\nabla (\rho_\ee \ast \widetilde\omega_\ee)] \cdot \nabla \varphi - \nabla (\rho_\ee \ast \widetilde\omega_\ee) \cdot S_\ee [\nabla \varphi ] \right) \dy \dx \dt \\
        & = I_{11}^{(1)} + I_{11}^{(2)} + I_{11}^{(3)}.
    \end{align*}
    First, for $I_{11}^{(1)}$ we take advantage of Lemma~\ref{lem:S properties}--\ref{S convergence} in order to obtain that
    \begin{align*}
         I_{11}^{(1)} & =  \int_0^T \iint \nabla | S_\ee [ \rho_{\ee} \ast \widetilde\omega_\ee] |^2 \cdot \nabla \varphi \dy \dx \dt =- \int_0^T \iint  | S_\ee [ \rho_{\ee} \ast \widetilde\omega_\ee] |^2 \Delta \varphi \dy \dx \dt \\
        & \rightarrow -  \int_0^T \int |\nabla \rho|^2 \Delta \varphi \dx \dt,
    \end{align*}
    where we use the convergence result from Proposition~\ref{lem:Frechet-Kolmogorov}. In order to study $I_{11}^{(2)}$ we change variables and we get that
    \begin{align*}
        I_{11}^{(2)} & =  \int_0^T \iint \omega_\ee (y) \frac{(\rho_\ee \ast \widetilde\omega_\ee) (x-y) - (\rho_\ee \ast \widetilde\omega_\ee) (x) }{\ee} \nabla (\rho_\ee \ast \widetilde\omega_\ee)(x) \cdot \frac{\nabla \varphi (x - y) - \nabla \varphi (x)}{\ee} \dy \dx \dt \\
        & =  \int_\Td \omega_1 (y) \int_0^T \int_\Td \frac{(\rho_\ee \ast \widetilde\omega_\ee) (x- \ee y) - (\rho_\ee \ast \widetilde\omega_\ee) (x) }{\ee} \nabla (\rho_\ee \ast \widetilde\omega_\ee)(x) \cdot \frac{\nabla \varphi (x - \ee y) - \nabla \varphi (x)}{\ee}  \dx \dt \dy .
    \end{align*}
    From \eqref{eq:strong H1 convolution} and  \cite[Lemma A.1]{Elbar_Skrzeczkowski23} we can show that for fixed $y \in \Td$ it follows that 
    \begin{align*}
        & \int_0^T \int_\Td \frac{(\rho_\ee \ast \widetilde\omega_\ee) (x- \ee y) - (\rho_\ee \ast \widetilde\omega_\ee) (x) }{\ee} \nabla (\rho_\ee \ast \widetilde\omega_\ee)(x) \cdot \frac{\nabla \varphi (x - \ee y) - \nabla \varphi (x)}{\ee}  \dx \dt \\
        & \qquad \rightarrow \int_0^T \int_\Td \nabla \rho (x) \cdot y \nabla \rho (x) \cdot ( D^2 \varphi (x) y ) \dx \dt .
    \end{align*}
    We apply the dominated convergence theorem with respect to $y$ with the dominating function $\sup_\ee \| \nabla (\rho_\ee \ast \widetilde\omega_\ee) \|_{L^\infty(0,T; L^2(\Td))}^2 |y|^2 \| D^2 \varphi \|_{L^\infty}$ due to \eqref{eq:Uniform H1 Section 4}. Hence, thanks to the definition of $\omega_1$ and the symmetry of $D^2 \varphi$ we get that
    \begin{align*}
        I_{11}^{(2)} & \rightarrow  \int_\Td \omega_1 (y) |y|^2 \dy \int_0^T \int_\Td \nabla \rho (x) \cdot D^2 \varphi (x) \nabla \rho (x) \dx \dt \\
        & =   \int_0^T \int_\Td ( \nabla \rho (x)  \otimes  \nabla \rho (x) ) : D^2 \varphi (x) \dx \dt.
    \end{align*}
    For $I_{11}^{(3)}$ we apply Lemma~\ref{lem:S properties}--\ref{S (ii)} and we recover that
    \begin{equation*}
        I_{11}^{(3)} = 2 \int_0^T \iint S_\ee [\rho_\ee \ast \widetilde\omega_\ee] \frac{\sqrt{\omega_\ee (y)}}{\sqrt{2} \ee} [ ( \nabla (\rho_\ee \ast \widetilde\omega_\ee) (x-y) - \nabla (\rho_\ee \ast \widetilde\omega_\ee) (x) ) \cdot ( \nabla \varphi (x-y) - \nabla \varphi (x) ) ] \dy \dx \dt.
    \end{equation*}
    Afterwards, we take advantage of the Cauchy-Schwarz inequality, the uniform bound \eqref{eq:Uniform H1_2 Section 4}, and a Taylor expansion, and we have that
    \begin{align*}
        |I_{11}^{(3)} | & \leq C \int_0^T \iint \frac{\omega_\ee (y)}{\ee^2} |  \nabla (\rho_\ee \ast \widetilde\omega_\ee) (x-y) - \nabla (\rho_\ee \ast \widetilde\omega_\ee) (x)  |^2 |  \nabla \varphi (x-y) - \nabla \varphi (x) |^2 \dy \dx \dt \\
        & \leq C \ee^2 \| D^2 \varphi \|_{L^\infty}^2 \int_0^T \iint \frac{\omega_\ee (y)}{\ee^2} |  \nabla (\rho_\ee \ast \widetilde\omega_\ee) (x-y) - \nabla (\rho_\ee \ast \widetilde\omega_\ee) (x)  |^2 \dy \dx \dt \rightarrow 0.
    \end{align*}
    In order to conclude the analysis on $I_1$, it just remains to study $I_{12}$. Analogously to $I_{11}$, we can find a decomposition into three different terms. In particular, we have that
    \begin{align*}
        I_{12} & = 2 \int_0^T \iint S_\ee [\rho_\ee \ast \widetilde\omega_\ee ]^2 \Delta \varphi \dy \dx \dt + 2 \int_0^T \iint S_\ee [\rho_\ee \ast \widetilde\omega_\ee] (\rho_\ee \ast \widetilde\omega_\ee ) S_\ee [\Delta \varphi ] \dy \dx \dt \\
        & \quad + 2 \int_0^T \iint S_\ee [\rho_\ee \ast \widetilde\omega_\ee ] \left( S_\ee [(\rho_\ee \ast \widetilde\omega_\ee) \Delta \varphi ] - S_\ee [\rho_\ee \ast \widetilde\omega_\ee ] \Delta \varphi - (\rho_\ee \ast \widetilde\omega_\ee) S_\ee [\Delta \varphi ] \right) \dy \dx \dt .
    \end{align*}
    Therefore, taking advantage of this decomposition, we can work analogously as we have done for $I_{11}$ in order to get that
    \begin{align*}
        I_{12} \rightarrow 2 \int_0^T \int |\nabla \rho (x) |^2 \Delta \varphi \dx \dt +   \int_0^T \int \rho (x) \nabla \rho (x) \cdot \nabla \Delta \varphi (x) \dx \dt .
    \end{align*}

    \textit{Step 2b. Controlling $I_2$.}
    We expand $I_2$ in order to recover
    \begin{align*}
        I_2 & = - \int_0^T \iiint \frac{\omega_\ee(y)}{\ee^2} \widetilde{\omega}_\ee (z) \big[ \rho_{\ee} (x-z) (\nabla \varphi (x-z) - \nabla \varphi (x)) \\
        & \hspace{45mm} + \rho_{\ee} (x-y-z) ( \nabla \varphi (x-y) - \nabla \varphi (x-y-z)) \big] \\
        & \hspace{42mm} \cdot \big[ \nabla (\rho_{\ee} \ast \widetilde{\omega}_\ee ) (x) - \nabla (\rho_{\ee} \ast \widetilde{\omega}_\ee ) (x-y) \big] \dz\dy\dx\dt.
    \end{align*}
    Hence,  to study $I_2$ we take advantage of the Taylor expansion
    \begin{align*}
        \nabla\varphi(x-y)-\nabla\varphi(x-y-z) & = \nabla\varphi(x)-\nabla\varphi(x-z) - yz D^{3}\varphi(x) + O(y^2 + yz^2 ) \\
        &  = \nabla\varphi(x)-\nabla\varphi(x-z) + O(y^2 + yz ).
    \end{align*}
    Therefore, we recover that $I_2 = I_{21} + I_{22}$ with 
    \begin{align*}
        I_{21} & = - \int_0^T \iiint \frac{\omega_\ee (y)}{\ee^2} \widetilde\omega_\ee (z) \big[ (\nabla \varphi (x-z) - \nabla \varphi (x)) (\rho_{\ee} (x-z) - \rho_{\ee} (x-y-z) ) \big] \\
        & \hspace{42mm} \cdot \big[ \nabla (\rho_{\ee} \ast \widetilde\omega_\ee ) (x) - \nabla (\rho_{\ee} \ast \widetilde\omega_\ee) (x-y) \big] \dz \dy \dx \dt
    \end{align*}
    and
    \begin{align*}
        I_{22} & = - \int_0^T \iiint \frac{\omega_\ee (y)}{\ee^2} \widetilde\omega_\ee (z) \rho_{\ee} (x-y-z)   O(y^2 + yz)  \\
        & \hspace{42mm} \cdot \big[ \nabla (\rho_{\ee} \ast \widetilde\omega_\ee) (x) - \nabla (\rho_{\ee} \ast \widetilde\omega_\ee ) (x-y) \big] \dz\dy\dx\dt.
    \end{align*}
    Let us start with $I_{21}$. Taking the corresponding Taylor expansions, we can rewrite it, and it follows that
    \begin{align*}
        I_{21} = \int_0^T \iiint & \sqrt{2} S_\ee [\nabla (\rho_{\ee} \ast \widetilde\omega_\ee)] (x,y) \\
        & \cdot \frac{\sqrt{\omega_\ee (y)}}{\ee} \widetilde\omega_\ee(z) \left( -z D^2 \varphi(x) + O(z^2) \right) \int_0^1 y \cdot \nabla\rho_{\ee} (x-sy-z) \ds \dz \dy \dx \dt.
    \end{align*}
    Inside the integral $O(yz) \leq C \ee \widetilde\ee$. This observation, combined with Young's  inequality, implies that 
    \begin{align*}
        I_{21} & \leq C\widetilde\ee \int_0^T \iiint \left| S_\ee [\nabla (\rho_{\ee} \ast \widetilde\omega_\ee)] (x,y) \right|^2 \dz\dy\dx\dt \\
        & \quad + C\widetilde\ee \int_0^T \iint \left( \int \widetilde\omega_\ee (z) \sqrt{\omega_\ee (y)} \int_0^1 |\nabla \rho_{\ee} (x- sy -z) | \ds\dz \right)^2 \dy\dx\dt   \\
        & = I_{21}^{(1)} + I_{21}^{(2)} .
    \end{align*}
    First we notice that $I_{21}^{(1)} \leq C \widetilde\ee \rightarrow 0$ due to \eqref{eq:Bounds S_ee}. For $I_{22}^{(2)}$, let us first take advantage of a convenient change of variables. Afterwards, we take advantage of Definition~\ref{def:Mollifying sequence} for the mollifier kernel and Jensen's inequality, and we get that
    \begin{align*}
       I_{21}^{(2)} & = C \widetilde\ee \int_0^T \iint \left( \sqrt{\omega_\ee (y)} \int \widetilde\omega_\ee (z) \int_0^1 |\nabla \rho_{\ee} (x-sy) | \ds \dz \right)^2 \dy \dx \\
       & = C \widetilde\ee \int_0^T \iint \left( \sqrt{\omega_\ee (y)}   \int_0^1 |\nabla \rho_{\ee} (x-sy) | \ds \right)^2 \dy \dx \\
       & \leq C \widetilde\ee \int_0^T \iint \omega_\ee (y)   \int_0^1 |\nabla \rho_{\ee} (x-sy) |^2 \ds   \dy \dx = C\widetilde\ee \| \nabla \rho_{\ee} \|_{L^2((0,T) \times \Td)}^2,
    \end{align*}
    and $I_{21}^{(2)} \leq C \frac{\widetilde\ee}{\ee^*} \rightarrow 0$ due to \eqref{eq:H1 bound} and the definition of $\ee^*$ in~\eqref{def:widetilde ee}. Note that this term, which was the most difficult one to treat, motivated the introduction of the parameters $\ee^*$ and $\alpha$.

    Let us now focus on $I_{22}$. Since $\ee \ll \widetilde\ee$, inside the integral $O(y^2 + yz) = O(yz) \leq C \ee \widetilde\ee$. Furthermore, if we also take advantage of Young's inequality, we recover
    \begin{align*}
        I_{22} & \leq C \widetilde\ee \int_0^T \iint \sqrt{2} \left| S_\ee [\nabla (\rho_{\ee} \ast \widetilde\omega_\ee) ] (x,y) \right| \sqrt{\omega_\ee (y)} (\rho_{\ee} \ast \widetilde\omega_\ee) (x-y)  \dy\dx\dt \\
        & \leq  C \widetilde\ee \int_0^T \iint \left| S_\ee [\nabla (\rho_{\ee} \ast \widetilde\omega_\ee)] (x,y) \right|^2 \dy\dx\dt + C \widetilde\ee \int_0^T \iint \omega_\ee (y) ( \rho_{\ee} \ast \widetilde\omega_\ee)^2 (x-y) \dy\dx\dt .
    \end{align*}
    For the first term on the RHS, we use \eqref{eq:Bounds S_ee}. For the second term, we take a change of variables, and we use Definition~\ref{def:Mollifying sequence} and \eqref{eq:Uniform Lp Section 4}.  Thus, it follows that
    \begin{align*}
        I_{22} & \leq C \widetilde\ee + C \widetilde\ee \int_0^T \int (\rho_{\ee} \ast \widetilde\omega_\ee)^2 (x) \int \widetilde\omega_\ee (y) \dy \dx \dt  \leq C \widetilde\ee \rightarrow 0. 
    \end{align*}

    \textit{Step 3. Convergence of the aggregation term $J$.} We study the convergence of the nonlocal term corresponding to the aggregation.
    \begin{equation}\label{eq:L-MG}
    \begin{split}
        J & = - 2 \int_0^T \int \dive \left( \rho_{\ee} \nabla \left( \widetilde{\omega}_\ee \ast \widetilde\omega_\ee \ast \rho_{\ee} \right) \right) \varphi  = 2 \int_0^T \int \widetilde\omega_\ee \ast (\rho_{\ee} \nabla \varphi ) \cdot \nabla (\widetilde\omega_\ee \ast \rho_{\ee})  \\
        & = \underbrace{2\int_0^T\int (\widetilde\omega_\ee \ast \rho_{\ee}) \nabla (\widetilde\omega_\ee \ast \rho_{\ee}) \cdot \nabla \varphi}_{=: \, J_1} + \underbrace{2 \int_0^T \int \left( \widetilde\omega_\ee \ast (\rho_{\ee} \nabla \varphi) - (\widetilde\omega_\ee \ast \rho_{\ee}) \nabla \varphi \right) \cdot \nabla ( \widetilde\omega_\ee \ast \rho_{\ee} )}_{=: \, J_2}.
        \end{split}
    \end{equation}
    Let us begin with the analysis of $J_1$. From Proposition~\ref{lem:Frechet-Kolmogorov} it follows that
    \begin{equation*}
         ( \widetilde\omega_\ee \ast \rho_{\ee})  \nabla (\widetilde\omega_\ee \ast \rho_{\ee})   \rightarrow  \rho \nabla\rho  \quad  \text{strongly in } L^{1}((0,T)\times \Td).
    \end{equation*}
    We proceed now to study the error term $J_2$,
    \begin{equation*}
        J_2 = 2 \int_0^T \iint \widetilde{\omega}_\ee (y) \rho_{\ee} (x-y) \left( \nabla \varphi (x-y) - \nabla \varphi (x) \right) \cdot \nabla ( \widetilde{\omega}_\ee \ast \rho_{\ee} ) (x) \, \diff y \diff x \diff t.
    \end{equation*}
    We recall $\widetilde{\omega}_\ee$ is supported on a ball of size $\widetilde{\ee}$. Furthermore, since $\varphi$ is smooth, it follows that by a Taylor expansion $| \nabla \varphi (x - y) - \nabla \varphi (x) | \leq C \widetilde{\ee}$. Hence, using again  Proposition~\ref{prop:Some properties}, it follows that
    \begin{align*}
        J_2 & \leq C \widetilde{\ee} \int_0^T \int \left| ( \widetilde{\omega}_\ee \ast \rho_{\ee} ) \nabla ( \widetilde\omega_\ee \ast \rho_{\ee} )   \right| \leq C \widetilde{\ee} \| \widetilde{\omega}_\ee \ast \rho_{\ee} \|_{L^{\infty}(0,T; L^{2} (\Td) )}  \| \nabla ( \widetilde{\omega}_\ee \ast \rho_{\ee} ) \|_{L^{\infty}(0,T; L^{2} (\Td) )} \\
        & \leq C \widetilde{\ee} \rightarrow 0 .
    \end{align*}

    \textit{Step 4. Convergence of the remainder $\mathfrak{R}$.}
    Let us integrate by parts and use H\"older's inequality in order to recover
    \begin{equation*}
        |\mathfrak{R}| = \int_0^T \int_\Td \ee^* \rho_\ee | \nabla \rho_\ee | | \nabla \varphi | \leq \| \sqrt{\ee^*} \rho_\ee \|_{L^2((0,T)\times\Td)} \| \sqrt{\ee^*} \nabla \rho_\ee \|_{L^2((0,T)\times\Td)} \| \nabla \varphi \|_{L^\infty (0,T; L^{\infty}(\Td))}.
    \end{equation*}
    From \eqref{eq:H1 bound} we know that $\| \sqrt{\ee^*} \nabla \rho_\ee \|_{L^2 ((0,T)\times \Td)} \leq C$ uniformly in $\ee^*$. Furthermore, from the Gagliardo-Nirenberg inequality, it follows that 
    \begin{equation*}
        \| \sqrt{\ee^*} \rho_\ee \|_{L^2 ((0,T)\times\Td)} \leq (\ee^*)^{1 - \frac{\theta}{2}} \| \sqrt{\ee^*} \nabla \rho_\ee \|_{L^2((0,T)\times\Td)}^\theta \| \rho_\ee \|_{L^1(\Td)}^{1-\theta}
    \end{equation*}
    for $\theta = \frac{d}{d+2}$ and in particular
    \begin{equation*}
        \| \sqrt{\ee^*} \rho_\ee \|_{L^2 (\Td)} \leq  (\ee^*)^{\frac{d + 4}{2d + 4}} C \rightarrow 0, \quad \text{ as $\ee\to 0$}.
    \end{equation*}

    \textit{Step 5. Conclusion.} 
    Combining all the steps we recover that when we take the limit $\alpha \rightarrow 0$ and $\ee \rightarrow 0$, immediately after we have that
    \begin{align*}
        & - \int_0^T \int_\Td \rho \frac{\partial \varphi}{\partial t} \dx \dt -  \int_\Td \rho_0 \varphi (0) \dx = \frac{1}{2} \int_0^T \int_\Td ( \nabla \rho (x) \otimes \nabla \rho (x) ) : D^2 \varphi (x) \dx \dt \\
        & \hspace{20mm} \quad + \frac{1}{2} \int_0^T \int_\Td |\nabla \rho (x) |^2 \Delta \varphi \dx \dt + \frac{1}{2} \int_0^T \int_\Td \rho (x) \nabla \rho (x) \cdot \nabla \Delta \varphi (x) \dx \dt \\
        & \hspace{20mm} \quad+ 2\int_0^T \rho(x) \nabla\rho (x) \cdot \nabla \varphi (x) \dx \dt  .
    \end{align*}
In fact, we have even more regularity on the solution: $\rho\in L^{\infty}(0,T; H^{1}(\T^d)) \cap L^{2}(0,T; H^{2}(\T^d))$. This follows from the uniform bounds on the free energy, ~\eqref{eq:Bounds S_ee} and~\cite[Theorem 1.2]{Ponce04}.
\end{proof}


\section{Convergence in the $2$-Wasserstein distance. Proof of \texorpdfstring{Theorem~\ref{thm:Nonlocal to local}}{Theorem 1.9}}\label{sec:Convergence W2}

In this section we want to show that we take the limit on the variables $\ee$ and $\alpha$ simultaneously. In order to do that, we take advantage of the $2$-Wasserstein distance. First, in Subsection \ref{sec:Vanishing viscosity term}  we perform a commutator estimate to deal with the vanishing viscosity term. To do that we extend the result by Amassad and Zhou in \cite{Amassad_Zhou25} to aggregation-diffusion equations. Afterwards, in Subsection \ref{sec:W2 through subsequence} we take advantage of previous results in order to give the proof of Theorem~\ref{thm:Nonlocal to local}. Thereby, we show convergence of $\rho_{\ee,\alpha}$, a solution of \eqref{eq:Nonlocal CH}, to a solution of \eqref{eq:CH} along a subsequence $(\ee_k, \alpha_k)$. We note that the theorem is only stated for $m=2$. Therefore, and until the end of the article, we assume $m=2$. In this section, we use the notations from Definition~\ref{def:vanishing_sequence}.

\subsection{The vanishing viscosity term}\label{sec:Vanishing viscosity term}

Let us fix $\ee > 0$ and let us denote
\begin{equation*}
    \bold w [\rho] =  \left(  B_\ee [\rho \ast \widetilde\omega_\ee \ast \widetilde\omega_\ee] - 2 \,  \widetilde\omega_{\ee} \ast\widetilde\omega_\ee \ast \rho  \right).
\end{equation*}
Let us consider $0 < \eta < \alpha$, we establish the commutator estimate for the intermediate scale $\rho_\eta = R_\eta \ast \rho_\alpha$. 
We also mention here a few auxiliary results. First, \cite[Lemma 3.2]{Amassad_Zhou25}.
\begin{lem}\label{lem:omega eta}
    Let $R$ be an admissible kernel in the sense of Definition~\ref{def:vanishing_sequence} and $0 < \eta < \alpha$. Then, there exists $C>0$ such that for all $f \in H^{-1}(\Td)$,
    \begin{equation}\label{eq:omega eta}
        \| R_\eta \ast f \|_{L^2(\Td)} \leq C \left( \frac{\alpha}{\eta} \right)^k \|  R^{\f 12}_\alpha \ast f \|_{L^2(\Td)}. 
    \end{equation}
\end{lem}

Then, we include \cite[Lemma 3.3]{Amassad_Zhou25} as well.
\begin{lem}\label{lem:Property mollifier}
    Let $R$ be an admissible kernel in the sense of Definition~\ref{def:vanishing_sequence} and $p \in [1, \infty)$. Then, for all nonnegative $f \in L^2(\Td)$, it follows that
    \begin{align}
        \left\| |\nabla R^{\f 12}_\alpha| \ast f \right\|_{L^2(\Td)} & \leq C \left( \frac{1}{\alpha} \right) \left\|  R^{\f 12}_\alpha \ast f \right\|_{L^2(\Td)} \label{eq:Property mollifier} 
    \end{align}
    where $C = \| h \|_{TV}$.
\end{lem}

Furthermore, we recall the displacement convexity of the internal energy~\cite[Proposition 7.29]{Santambrogio15}.
\begin{lem}[Displacement convexity]
    \label{lem:Convexity}
    The energy functional $\me_2[\rho]$ is displacement convex.
\end{lem}

The next auxiliary result shows how we should take $\alpha$ with respect to $\ee$ to prove our main theorem. This proposition is an extension of the result \cite[Theorem 1.2]{Amassad_Zhou25} where we also include an aggregation term.

\begin{prop}
    [Commutator estimate]
    Take $\ee > 0$ fixed. Under the assumptions of Theorem~\ref{thm:Nonlocal to local}, take $\rho_\alpha$ a solution of \eqref{eq:Nonlocal CH} for $\alpha > 0$ and $\rho$ a solution of \eqref{eq:Nonlocal CH} for $\alpha = 0$. Then there exists $\gamma > 0$ such that if we take $\eta = \alpha^{1 + \gamma}$ and $\rho_\eta = R_\eta\ast \rho_\alpha$, there exists constants $r$, $k'$, $C > 0$, independent of $\alpha$, and $\ee$ such that
    \begin{equation*}
        \sup_{t\in[0,T]}\mathcal{W}_2^2 (\rho, \rho_\eta) (t)\le C \frac{\alpha^r}{\ee^{k'}} \exp{\left( \frac{T}{\ee^{k'}} \right)}.
    \end{equation*} 
    \label{prop:commutator estimate}
\end{prop}


\begin{proof}
    We split the proof in several steps.

    \textit{Step 0. The intermediate scale.} We notice that 
    \begin{equation*}
        \partial_t \rho_\eta = \dive \left( \frac{ R_\eta \ast (\rho_\alpha \nabla (\ee^\ast  R_\alpha \ast \rho_\alpha + \bold w [\rho_\alpha]))}{\rho_\eta} \rho_\eta \right) .
    \end{equation*}
    We can apply Lemma~\ref{lem:action-minimising path} to $\rho$ and $\rho_\eta$. Let $(\varphi^t, \psi^t)$ be their corresponding Kantorovich potentials. We have
    \begin{align*}
        \frac{\diff}{\diff t} \left[ \frac{1}{2} \mathcal{W}_2^2 (\rho, \rho_\eta) \right] & = \int_\Td \left( \nabla\varphi^t \cdot (-  R_\eta \ast (\rho_\alpha \nabla (\ee^\ast  R_\alpha \ast \rho_\alpha + \bold w [\rho_\alpha] ) )) + \nabla\psi^t \cdot ( - \rho \nabla (\ee^\ast \rho + \bold w[\rho])) \right) \dx \\
        & = \underbrace{\ee^\ast \int_\Td \left( \nabla\varphi^t \cdot (-  R_\eta \ast (\rho_\alpha \nabla ( R_\alpha \ast \rho_\alpha ))) + \nabla\psi^t \cdot ( - \rho \nabla \rho ) \right) \dx}_{\eqqcolon \mathcal{D}_{\rho, \rho_\eta}} \\
        & \quad + \underbrace{\int_\Td \left( \nabla\varphi^t  \cdot (-  R_\eta \ast (\rho_\alpha \nabla \bold w[\rho_\alpha])) + \nabla\psi^t \cdot ( - \rho \nabla  \bold w[\rho]) \right) \dx}_{\eqqcolon \mathcal{V}_{\rho, \rho_\eta}} .
    \end{align*}

     \textit{Step 1. The viscosity diffusion term $\mathcal{D}_{\rho, \rho_\eta}$.}
    The viscosity diffusion term $\mathcal{D}_{\rho, \rho_\eta}$ is the same in~\cite{Amassad_Zhou25}. Following verbatim their proof, which uses Lemmas~\ref{lem:omega eta}, Lemma~\ref{lem:Property mollifier} and Lemma~\ref{lem:Convexity} we obtain  $\mathcal{D}_{\rho, \rho_\eta} = G_{\rho, \rho_\eta}^\mathcal{D} - \ee^* C_{\rho,\rho_\eta}^{1,2,3}$ where
    $$
    G_{\rho, \rho_\eta}^\mathcal{D}\le 0,\quad \left|C_{\rho,\rho_\eta}^{1,2,3}\right|\le   C  \left(\frac{\eta}{\alpha}  +  \alpha^{\frac 1 p}\left( \frac{\alpha}{\eta} \right)^{2k}\right)
    $$
    for some $p>0$.

    \textit{Step 2. The velocity term $\mathcal{V}_{\rho, \rho_\eta}$.}
    We compute in order to recover
    \begin{align*}
        &\int_\Td \nabla\varphi^t \cdot ( R_\eta \ast (\rho_\alpha \nabla \bold w [\rho_\alpha])) \dx = \int_\Td \nabla\varphi^t \cdot ( R_\eta \ast \rho_\alpha) \nabla \bold w [ \rho_\alpha ] \dx + C_{\rho, \rho_\eta}^{(4)} \\
        & \qquad  = \int_{\Td} \nabla\varphi^t \cdot \rho_\eta \nabla \bold w [\rho_\eta] \dx  + C_{\rho, \rho_\eta}^{(4)} + \underbrace{\int_\Td \nabla\varphi^t \cdot \left( \rho_\eta \nabla (B_\ee[\rho_\alpha \ast \widetilde\omega_\ee \ast \widetilde\omega_\ee] - B_\ee[\rho_\eta \ast \widetilde\omega_\ee \ast \widetilde\omega_\ee]) \right) \dx}_{\eqqcolon \, C_{\rho, \rho_\eta}^{(5)}} \\
        &\qquad \qquad \quad + \underbrace{2\int_\Td \nabla\varphi^t \cdot \left( \rho_\eta \nabla \left(  \widetilde\omega_\ee  \ast \widetilde\omega_\ee\ast  (\rho_\eta - \rho_\alpha )\right) \right) \dx}_{\eqqcolon \, C_{\rho, \rho_\eta}^{(6)}}.
    \end{align*}
    Therefore, analogously to the previous step, we achieve a gradient flow structure,
    \begin{equation*}
        G_{\rho, \rho_\eta}^{\mathcal{V}} \coloneqq \int_\Td (\nabla\varphi^t \cdot (- \rho_\eta \nabla \bold w [\rho_\eta]) + \nabla\psi^t \cdot (- \rho \nabla \bold w [\rho]))
    \end{equation*}
    and $\mathcal{V}_{\rho, \rho_\eta}$ can be rewritten as 
    \begin{equation*}
        \mathcal{V}_{\rho, \rho_\eta} = G_{\rho, \rho_\eta}^{\mathcal{V}} - (C_{\rho, \rho_\eta}^{(4)} + C_{\rho, \rho_\eta}^{(5)} + C_{\rho, \rho_\eta}^{(6)} ).
    \end{equation*}
    However, the term $G_{\rho, \rho_\eta}^{\mathcal{V}}$ is not nonpositive, as the associated energy may not be displacement convex. We first focus on bounding the three commutators.

    \textit{Step 2a. Bound of $C_{\rho, \rho_\eta}^{(4)}$.} We define
    \begin{align*}
        C_{\rho, \rho_\eta}^{(4)} & \coloneqq \int_\Td \nabla\varphi^t(x) \cdot \left[  R_\eta \ast (\rho_\alpha \nabla \bold w[\rho_\alpha]) - (  R_\eta \ast \rho_\alpha) \nabla \bold w [\rho_\alpha] \right] \dx \\
        & = \iint \nabla\varphi^t (x) \cdot  R_\eta (y) \rho_\alpha (x-y) ( \nabla \bold w [\rho_\alpha] (x) - \nabla \bold w [\rho_\alpha] (x-y)) \dy \dx .
    \end{align*}
    After a Taylor expansion 
    \begin{align*}
        C_{\rho, \rho_\eta}^{(4)} & = \iint \nabla\varphi^t (x) \cdot ( R_\eta ) (y) \rho_\alpha (x-y) (D^2\bold w [\rho_\alpha] (y) y+ O(y^2)) \dy \dx .
    \end{align*}
    Hence,
    \begin{align*}
        | C_{\rho, \rho_\eta}^{(4)} | & \leq C\| \nabla\varphi^t \|_{L^\infty(\Td)} (1+\|D^2\bold w[\rho_\alpha]\|_{L^{\infty}(\Td)})\left( \int_\Td |y| ( R_\eta ) (y) \dy  \right)  \\
        & \leq C \eta,
    \end{align*}
    where $C$ depends polynomially on $\ee^{-1}$  (since $\widetilde{\ee}$ and $\ee^*$ also depends on $\ee$). 
   
    \textit{Step 2b. Bound of $C_{\rho, \rho_\eta}^{(5)}$ and  $C_{\rho, \rho_\eta}^{(6)}$.} 
    We apply a Taylor expansion in order to recover that
    \begin{align*}
        | C_{\rho, \rho_\eta}^{(5)} | & = \left| \int_\Td \nabla\varphi^t(x) \cdot \left( \rho_\eta (x) \nabla \left( B_\ee [ R_\eta \ast \rho_\alpha \ast \widetilde\omega_\ee \ast \widetilde\omega_\ee] - B_\ee [\rho_\alpha \ast \widetilde\omega_\ee \ast \widetilde\omega_\ee] \right) \right) \dx \right| \\
        & = \left| \int_\Td \nabla\varphi^t(x) \cdot \left( \rho_\eta (x)  \left(  R_\eta \ast B_\ee [\rho_\alpha \ast \widetilde\omega_\ee \ast \nabla \widetilde\omega_\ee] - B_\ee [\rho_\alpha \ast \widetilde\omega_\ee \ast \nabla \widetilde\omega_\ee] \right) \right) \dx \right|\\
        & = \left|\iint \nabla\varphi^t (x) \cdot \left( \rho_\eta (x)  R_\eta (y) \left( B_\ee [\rho_\alpha \ast \widetilde\omega_\ee \ast \nabla \widetilde\omega_\ee] (x-y) - B_\ee [\rho_\alpha \ast \widetilde\omega_\ee \ast \nabla \widetilde\omega_\ee](x) \right) \right) \dx \right|.
    \end{align*}
    We can proceed in the same way we did for \textit{Step 2a}, and thanks to this it follows the estimate 
    \begin{equation*}
        |C_{\rho, \rho_\eta}^{(5)}| \leq C \eta 
    \end{equation*}
    for some $C$ depending polynomially on $\ee^{-1}$.
    The proof for $C_{\rho, \rho_\eta}^{(6)}$ is the same. Hence, 
    \begin{equation*}
        |C_{\rho, \rho_\eta}^{(6)}| \leq C   \eta.
    \end{equation*}

    \textit{Step 3. Gradient flow structure.} 
    It only remains to treat the term
   \begin{equation*}
        G_{\rho, \rho_\eta}^{\mathcal{V}} \coloneqq \int_\Td (\nabla\varphi^t \cdot (- \rho_\eta \nabla \bold w [\rho_\eta]) + \nabla\psi^t \cdot (- \rho \nabla \bold w [\rho])).
    \end{equation*} 
    By definition of $\bold w$, we note that it is similar to treating
     \begin{equation*}
        G =  \int_\Td (\nabla\varphi^t \cdot (- \rho_\eta \nabla W_{\ee}\ast \rho_\eta ) + \nabla\psi^t \cdot (- \rho \nabla W_\ee\ast \rho))
    \end{equation*}  
    where $W_\ee$ is a smooth mollifier depending on the parameter $\ee$. By properties of the Kantorovich potentials we have $\nabla\varphi^t(x) = x- T^t(x)$, $\nabla\psi^t(x) = x - S^t(x)$ where $T^t$, $S^t$ represent the optimal transport maps from $\rho_\eta(t)$ to $\rho(t)$ and $\rho(t)$ to $\rho_\eta(t)$ respectively. Thus $T_t\sharp \rho_\eta=\rho$ and  $\nabla\psi^t\circ T = -\nabla\varphi^t$ since $S\circ T = id$.

    We obtain 
    \begin{align*}
    \int_{\T^d}\nabla\psi^t\cdot (- \rho \nabla W_\ee\ast \rho)) &= \int_{\T^d}\nabla\varphi^t\rho_\eta(\nabla W_\ee\ast\rho)\circ T^t\\
    &= \int_{\T^d}\int_{\T^d}\nabla\varphi^t(x)\rho_{\eta}(x)\nabla W_\ee (T^t(x)-y)\rho(y)\diff x \diff y\\
    &= \int_{\T^d}\int_{\T^d}\nabla\varphi^t(x)\rho_{\eta}(x)\nabla W_\ee (T^t(x)-T^t(y))\rho_\eta(y)\diff x \diff y.
    \end{align*}
    Therefore 
    $$
    G = \int_{\T^d}\int_{\T^d}\nabla\varphi^t(x)\rho_{\eta}(x)\left[\nabla W_\ee (T^t(x)-T^t(y))-\nabla W_\ee(x-y)\right]\rho_\eta(y)\diff x \diff y.
    $$
    Since $\|D^2 W_\ee\|_{L^{\infty}}\le C$ for some $C$ depending polynomially on $\ee^{-1}$. From the combination of the definitions of the Kantorovich potential and the Wasserstein distance (see Lemma~\ref{lem:action-minimising path}), and Jensen's inequality, we observe that 
    \begin{align*}
    |G| &\le C\int_{\Td}\int_{\T^d}|\nabla\varphi^t(x)|\rho_{\eta}(x)\rho_\eta (y)\left[|\nabla\varphi^t(x)| + |\nabla\varphi^t(y)|\right]\\
    &\le C\mathcal{W}_2^2(\rho(t),\rho_\eta(t)).
    \end{align*} 
    Hence, we summarise all the bounds, and we have that
    \begin{align*}
        \frac{1}{2} \mathcal{W}_2^2 (\rho, \rho_\eta) (t) - \frac{1}{2} \mathcal{W}_2^2 (\rho, \rho_\eta) (0) \leq \int_0^t - \ee^\ast  C_{\rho, \rho_\eta}^{(1,2,3)}  - ( C_{\rho, \rho_\eta}^{(4)} + C_{\rho, \rho_\eta}^{(5)} + C_{\rho, \rho_\eta}^{(6)}) \diff \tau + C\int_{0}^{t}W_{2}^{2}(\rho,\rho_\eta)(\tau)\diff \tau .
    \end{align*}
     Let us choose $\eta = \alpha^{1+\gamma}$ with $0 < \gamma < \frac{1}{2pk}$. Then, we find that there exists $r > 0$  and $C$ depending polynomially on $\ee^{-1}$ such that
    \begin{equation*}
        \frac{1}{2} \mathcal{W}_2^2 (\rho, \rho_\eta) (t) - \frac{1}{2} \mathcal{W}_2^2 (\rho, \rho_\eta) (0) \leq C \alpha^r + C\int_{0}^{t}W_{2}^{2}(\rho,\rho_\eta)(\tau)\diff \tau.
    \end{equation*}
    By Gronwall's lemma we obtain that there exists $k'>0$ and $C$ independent of $\ee$ and $\alpha$ such that for all $t\in[0,T]$ we have that
    \begin{equation*}
    \mathcal{W}_2^2 (\rho, \rho_\eta) (t)-  \mathcal{W}_2^2 (\rho, \rho_\eta) (0) \le C \frac{\alpha^r}{\ee^{k'}}\exp{\left( \frac{T}{\ee^{k'}} \right)} . \qedhere
    \end{equation*}
\end{proof}

From the result described in Proposition~\ref{prop:commutator estimate}, we can also recover the following convergence result that connects the nonlocal problem with and without the viscosity term. 
\begin{cor}\label{cor:wass_conv}
  There exists constants $r$, $k'$, $C > 0$, independent of $\alpha$, and $\ee$ such that
    \begin{equation*}
        \sup_{t\in[0,T]}\mathcal{W}_2 (\rho, \rho_\alpha) (t)\le C \frac{\alpha^{\frac r 2}}{\ee^{\frac {k'}{2}}}\exp{\left( \frac{T}{2\ee^{k'}}\right)}
    \end{equation*}
    where $\rho_\alpha$ is a solution of \eqref{eq:Nonlocal CH} for $\alpha > 0$ and $\rho$ is a solution of \eqref{eq:Nonlocal CH} for $\alpha = 0$. 
\end{cor}

\begin{proof}
    By the triangle inequality and Proposition~\ref{prop:commutator estimate}
    \begin{align*}
        \mathcal{W}_2 (\rho, \rho_\alpha) (t) & \leq \mathcal{W}_2 (\rho, \rho_\eta) (t) + \mathcal{W}_2 (\rho_\eta, \rho_\alpha) (t) \\
        & \leq \mathcal{W}_2 (\rho, \rho_\eta) (0) +\frac{\alpha^{\frac r 2}}{\ee^{\frac {k'}{2}}} \exp{\left( \frac{T}{2\ee^{k'}}\right)} + \mathcal{W}_2 (\rho_\eta, \rho_\alpha) (t).
    \end{align*}
    Note that in particular $\rho (0) = \rho_\alpha (0)$. Therefore, we have that
    \begin{equation*}
        \mathcal{W}_2 (\rho, \rho_\alpha) (t) \leq \frac{\alpha^{\frac r 2}}{\ee^{\frac {k'}{2}}}\exp{\left( \frac{T}{2\ee^{k'}}\right)} + \mathcal{W}_2 (\rho_\eta, \rho_\alpha) (0) + \mathcal{W}_2 (\rho_\eta, \rho_\alpha) (t).
    \end{equation*}
    Let us now recall that $\rho_\eta =  R_\eta \ast \rho_\alpha$. We define 
    \begin{equation*}
        \pi (x,y) =  R_\eta (x-y) \rho_\alpha (y).
    \end{equation*}
    Therefore, it is easy to verify that
    \begin{equation*}
        \int_\Td \pi (x,y) \dx = \rho_\alpha (y), \qquad \int_\Td \pi (x,y) \dy = \rho_\eta (x)
    \end{equation*}
    and hence $\pi$ is a transport plan between $\rho_\eta$ and $\rho_\alpha$. Thus, we have that
    \begin{equation*}
        \int_\Td |x-y|^2 \pi(x,y) \dx \dy = \int_\Td |x|^2  R_\eta (x) \dx \int_\Td \rho_\alpha (y) \dy \leq C \eta^2.
    \end{equation*}
    Therefore, it follows that, up to changing $\eta$,
    \begin{equation*}
        \mathcal{W}_2 (\rho_\eta , \rho_\alpha ) \leq C \eta \ll \frac{\alpha^{\frac r 2}}{\ee^{\frac {k'}{2}}} \exp{\left( \frac{T}{2\ee^{k'}}\right)},
    \end{equation*}
    from which we recover the desired result.
\end{proof}

\subsection{The limit along a subsequence}\label{sec:W2 through subsequence}

First from~Proposition~\ref{thm:Nonlocal to local H1} we deduce the following lemma:

\begin{lem}
There exists a subsequence $\ee_k$ such that for all $\varphi\in L^\infty((0,T)\times\T^d)$:
$$
\int_{0}^{T}\int_{\T^d}\rho_{\ee_k,0}\, \varphi \to \int_{0}^{T}\int_{\T^d}\rho\, \varphi,
$$
where $\rho$ is a weak solution of~\eqref{eq:CH}. 
\end{lem}

Let us now introduce the following auxiliary result, which can be easily proved by the Kantorovich-Rubinstein duality of the Wasserstein distance.

\begin{lem}\label{lem:W_2_Kanto_Rubin}
Let $\varphi\in L^\infty(0,T; \text{Lip} (\T^d))$ and $f,g$ smooth enough. Then 
$$
\int_{0}^{T}\int_{\T^d}\varphi(f-g)\le C \int_{0}^{T} \mathcal{W}_2(f(t),g(t))\diff t
$$
where $C$ depends only on $\|\varphi\|_{L^{\infty}(0,T; Lip(\T^d))}$.
\end{lem}

Now we are ready to prove the main result of this section.

\begin{proof}[Proof of Theorem~\ref{thm:Nonlocal to local}]
   From the triangular's inequality and Lemma~\ref{lem:W_2_Kanto_Rubin} we get that  for all $\varphi\in L^\infty(0,T; \text{Lip}(\T^d))$,
    \begin{equation*}
         \int_{0}^{T}\int_{\T^d}(\rho_{\eps_k,\alpha_k}-\rho)\, \varphi \le C\int_{0}^{T} \mathcal{W}_2(\rho_{\eps_k,\alpha_k}(t), \rho_{\ee_k,0}(t)) + \int_{0}^{T}\int_{\T^d}(\rho_{\ee_k,0}-\rho)\varphi.
    \end{equation*}
    From Corollary~\ref{cor:wass_conv} we have that
    \begin{equation*}
        \mathcal{W}_2 (\rho_{\ee_k , \alpha_k}, \rho_{\ee_k, 0} )\leq C \frac{\alpha_k^{\frac r 2}}{\ee_k^{\frac {k'}{2}}} \exp{\left( \frac{T}{2\ee_k^{k'}}\right)}  .
    \end{equation*}
    Therefore, it follows that
    \begin{equation*}
         \int_{0}^{T}\int_{\T^d}(\rho_{\eps_k,\alpha_k}-\rho)\, \varphi \le C \frac{\alpha_k^{\frac r 2}}{\ee_k^{\frac {k'}{2}}} \exp{\left( \frac{T}{2\ee_k^{k'}}\right)}   + \int_{0}^{T}\int_{\T^d}(\rho_{\ee_k,0}-\rho)\varphi.
    \end{equation*}    
    Hence, if we choose the ratio between $\alpha_k$ and $\ee_k$ adequately and for the first term on the RHS the result follows. The convergence can then be upgraded to a narrow convergence by convergence of the masses. 
\end{proof}


\section{Convexity and particle approximation. Proof of \texorpdfstring{Theorem~\ref{thm:Particle approximation}}{Thoerem 1.11}}\label{sec:Convexity}

In this section we provide a deterministic particle approximation in view of the nonlocal approximation \eqref{eq:Nonlocal CH} extending the so-called \textit{blob method}~\cite{Craig_Elamvazhuthi_Haberland_Turanova23}. We need to prove that the energy functional $\mathcal{F}_{\ee,\alpha}$ is $\lambda$ convex. Instead of using the \textit{above the tangent inequality} \cite{Craig17}, we use Lemma~\ref{lem:lambda_convex_interaction} in order to recover $\lambda$ convexity of our functional. 

\begin{prop}
The functional $\mathcal{F}_{\ee,\alpha}$ is $\lambda_{\ee,\alpha}$-geodesically convex with
$$
    \lambda_{\ee,\alpha} \simeq \left(\ee^{-2} \widetilde\ee^{-d-2} + \widetilde\ee^{-d-2} + \ee^\ast \alpha^{-d-2}\right).
$$
for some $C>0$ and independenf of $\ee,\alpha$.
\end{prop}

\begin{proof}
Let us remark that by the symmetry properties of the kernels,
\begin{align*}
\mathcal{F}_{\ee,\alpha}[\rho]&= \frac{1}{4} \Dee [\rho \ast \widetilde\omega_\ee] - \me_2 [ \rho \ast \widetilde\omega_\ee] + \me_2 [\ee^\ast \rho \ast  R^{\f 12}_\alpha]\\
&= \int_{\T^d}\rho B_{\eps}[\rho\ast\widetilde\omega_\ee\ast \widetilde\omega_\ee ]-\int_{\T^d}\rho\, (\widetilde\omega_\ee\ast\widetilde\omega_\ee\ast \rho) + \ee^*\int_{\Td}\rho\, (R_\alpha \ast \rho)\\
&= \int_{\T^d}\rho\, ( \mathfrak{W}_{\ee,\alpha}\ast\rho)
\end{align*}

where 
$$
\mathfrak{W}_{\ee,\alpha}=\frac{\widetilde\omega_\ee\ast\widetilde\omega_\ee - \omega_\ee \ast\widetilde\omega_\ee\ast\widetilde\omega_\ee}{\ee^2} -\widetilde\omega_\ee\ast\widetilde\omega_\ee + \ee^*R_{\alpha}.
$$
Since all the kernels have bounded second derivatives, we can compute in order to recover that
\begin{align*}
    \| (D^2 \mathfrak{W}_{\ee, \alpha} )_- \|_{L^\infty} & \leq \ee^{-2} \| D^2 \widetilde\omega_\ee\ast\widetilde\omega_\ee \|_{L^\infty} + \ee^{-2} \| D^2 \omega_\ee \ast\widetilde\omega_\ee\ast\widetilde\omega_\ee \|_{L^\infty} + \| D^2 \widetilde\omega_\ee\ast\widetilde\omega_\ee \|_{L^\infty} + \ee^\ast \| D^2 R_\alpha \|_{L^\infty} \\
    & \leq \ee^{-2} \| D^2 \widetilde\omega_\ee \|_{L^\infty} (\| \widetilde\omega_\ee \|_{L^1} + \| \omega_\ee \|_{L^1} \| \widetilde\omega_\ee \|_{L^1}) +  \| D^2 \widetilde\omega_\ee \|_{L^\infty} \| \widetilde\omega_\ee \|_{L^1} + \ee^\ast \| D^2 R_\alpha \|_{L^\infty}  \\
    & \leq C (\ee^{-2} \widetilde\ee^{-d-2} + \widetilde\ee^{-d-2} + \ee^\ast \alpha^{-d-2}). 
\end{align*}
Thus, we have that $\mathfrak{W}_{\ee,\alpha}$ is $\lambda_{\ee,\alpha}$ convex with 
$$
    \lambda_{\ee,\alpha} \simeq  \left(\ee^{-2} \widetilde\ee^{-d-2} + \widetilde\ee^{-d-2} + \ee^\ast \alpha^{-d-2}\right).
$$
Applying Lemma~\ref{lem:lambda_convex_interaction} yields the result.
\end{proof}

Moreover, in view of Theorem~\ref{thm:Nonlocal to local}, we know that along a subsequence $(\ee_k, \alpha_k)$ we have convergence in the $2$-Wasserstein distance and that $\alpha_k = \alpha_k (\ee_k)$.
This information is enough to show the existence of a unique gradient flow of $\mf_{\ee_k, \alpha_k}$ for fixed $\ee_k > 0$. It can be done following the theory in \cite{Ambrosio_Gigli_Savare08} and \cite[Section 5]{Carrillo_Craig_Patacchini19}. We do not provide the details here, but we refer to \cite{Craig_Elamvazhuthi_Haberland_Turanova23} to the interested reader for similar computations and to \cite{Carrillo_Esposito_Wu23, Carrillo_Esposito_Skrzeczkowski_Wu24} for further examples. 

In our setting, we consider \eqref{eq:Nonlocal CH} as a continuity equation where the velocity is given by $\bold v_{\ee, \alpha}$. Therefore, under mild assumptions on $\omega_1$, the empirical measure $\rho_{\ee_k, \alpha_k}^N(t) = \frac{1}{N} \sum_{i = 1}^N \delta_{X_{\ee_k, \alpha_k}^i (t)}$ is a weak solution to \eqref{eq:Nonlocal CH} provided the particles satisfy the following equation of motion
\begin{align*}
    \dot{X}^{i}_{\ee, \alpha}(t)  & = - \f{1}{N}\sum_{j=1}^{N}\nabla W_\ee(X^i_{\ee, \alpha} - X^j_{\ee, \alpha})  + 2\frac{1}{N}\sum_{j=1}^{N} \nabla \widetilde{\omega}_{\eps}\ast\widetilde{\omega}_{\eps}(X^i_{\ee, \alpha} - X^j_{\ee, \alpha}) \\
    & \quad -  \ee_k^*\frac{1}{N} \sum_{j=1}^{N} \nabla R_{\alpha_k}(X^i_{\ee, \alpha} - X^j_{\ee, \alpha})
\end{align*}
where $$
W_\ee = \f{\widetilde{\omega}_{\eps}\ast\widetilde{\omega}_{\eps} - \omega_{\eps}\ast\widetilde{\omega}_{\eps}\ast\widetilde{\omega}_{\eps}}{\eps^2}.
$$
Thus, as a consequence of the usual stability estimate for $\lambda$-gradient flows \cite[Theorem 11.2.1]{Ambrosio_Gigli_Savare08}, we know
\begin{equation*}
    \mathcal{W}_2 (\rho_{\ee_k, \alpha_k}^{N} (t) , \rho_{\ee_k , \alpha_k} (t) ) \leq e^{  \lambda_{\ee_k,\alpha_k} t} \mathcal{W}_2 (\rho_{\ee_k, \alpha_k }^{N} (0) , \rho (0)).
\end{equation*}
We choose $N = N(k) = o \left(e^{\lambda_{\ee_k, \alpha_k}} \right) \rightarrow +\infty$ as $k \rightarrow \infty$ at the correct speed such that
\begin{equation*}
    \lim_{k \rightarrow \infty}  \mathcal{W}_2 (\rho^{N (k)}_{\ee_k, \alpha_k} (0) , \rho (0))  = 0
\end{equation*}
according to assumption \eqref{eq:Initial Assumption}. 
Therefore, by the triangular inequality and Lemma~\ref{lem:W_2_Kanto_Rubin} we get that  for all $\varphi\in L^\infty(0,T; \text{Lip}(\T^d))$,
    \begin{align*}
         \int_{0}^{T}\int_{\T^d}(\rho_{\eps_k,\alpha_k}^{N(k)}-\rho)\, \varphi  
         & \le C \int_{0}^{T} \mathcal{W}_2(\rho_{\eps_k,\alpha_k}^{N(k)}(t), \rho_{\ee_k,\alpha_k}(t)) + C\int_{0}^{T} \mathcal{W}_2(\rho_{\eps_k,\alpha_k}(t), \rho_{\ee_k,0}(t)) \\
         & \quad + \int_{0}^{T}\int_{\T^d}(\rho_{\ee_k,0}-\rho)\varphi.
    \end{align*} 
    Therefore, if we combine this results with Theorem~\ref{thm:Nonlocal to local} and Corollary~\ref{cor:wass_conv}  it follows that $\rho^{N} = \rho_{\ee_k, \alpha_k}^{N(k)}$
\begin{equation*}
    \rho^{N} \rightarrow \rho \quad \text{narrowly on } (0,T) \times \Td.
\end{equation*}
and we infer Theorem~\ref{thm:Particle approximation}.

\section*{Acknowledgements}
This work was supported by the European Union via the ERC AdG 101054420 EYAWKAJKOS project. AFJ was also partially funded via the ERC AdG 883363 Nonlocal-CPD (Nonlocal PDEs for Complex Particle Dynamics: Phase Transitions, Patterns and Synchronization). The authors are thankful to Antonio Esposito (University of L'Aquila), José Antonio Carrillo (University of Oxford) and Filippo Santambrogio (Université Claude Bernard Lyon 1) for useful conversations on the project.

\appendix
\section{Numerical simulations}\label{sec:Simulations}

We provide some numerical simulations of the model at the particle level to give a brief idea of its evolution. We show in comparison the associated equation at the local level. Concerning the particle simulations, we use the Sisyphe package that we modify for our purpose~\cite{Diez2021}. The SiSyPHE library simulates efficiently interacting particle systems, both on the GPU and on the CPU. It uses PyTorch and the KeOps library.

The equation of motion for the particles are the following: for $i=1,\ldots, N$
$$
\dot{X}_{i}(t)  = - \f{1}{N}\sum_{j=1}^{N}\nabla W_\ee(X_i - X_j) + \frac{m}{m-1}\sum_{j=1}^{N} \nabla \widetilde{\omega}_{\eps}(X_i-X_j) \left(\f{1}{N}\sum_{k=1}^{N}\widetilde{\omega}_{\eps}(X_j - X_k)\right)^{m-1}
$$
where 
$$
    W_\ee = \f{\widetilde{\omega}_{\eps}\ast\widetilde{\omega}_{\eps} - \omega_\eps\ast\widetilde{\omega}_{\eps}\ast\widetilde{\omega}_{\eps}}{\eps^2}.
$$
For the simulations in Figure~\ref{fig:Numerics}, the mollifiers $\omega_\eps$ and $\widetilde{\omega}_{\eps}$ are  Gaussian kernels with variance  $\eps$ and $\widetilde\eps$ respectively.
Note that in the numerical simulations we neglect the artificial viscosity term $\eps^\ast\dive(\rho\nabla R_{\alpha}\ast\rho)$ which it is only useful for the analysis and is small compared to the other terms. The particles are uniformly distributed in a grid and can overlap during the evolution. 

We compare the numerical simulations at the particle level with the corresponding local PDE given by
$$
\p_t \rho  + \dive(\rho\nabla\Delta\rho)  + \Delta \rho ^m=0.
$$
This equation is stabilized using the SAV (\textit{scalar auxiliary variable}) method~\cite{Jie-2018-SAV}. More precisely, a variant designed for degenerate parabolic models that preserves the physical bounds of the solution~\cite{Fukeng-2021-bounds,huang-SAV-fourth}. The SAV method allows solving efficiently (and also linearly) the equation while preserving the dissipation of a modified energy. A detailed scheme can be found in the article of the first author~\cite{Elbar_Poulain24}, where a compressible Cahn-Hilliard-Navier-Stokes model is considered. Here, we only retain the scheme of the Cahn-Hilliard part.


\begin{figure}[H]
    \centering
    \begin{minipage}{0.415\textwidth}
        \centering
        \includegraphics[width=\linewidth, height=5.45cm]{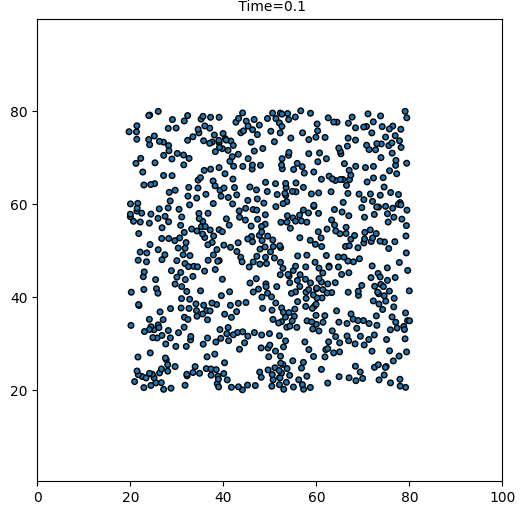}
        \caption*{Particles -- Initial datum.}
        \label{fig:3}
    \end{minipage}
    \hfill
    \begin{minipage}{0.415\textwidth}
        \centering
        \includegraphics[width=\linewidth, height=5.5cm]{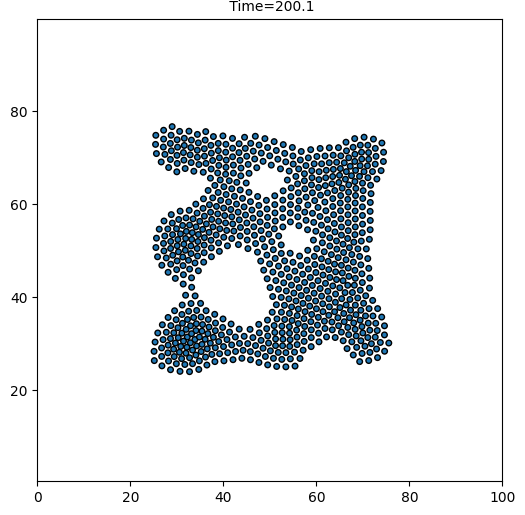}
        \caption*{Particles  -- Evolution.}
        \label{fig:4}
        
    \end{minipage}
    \hspace{0.9cm}
    \vskip\baselineskip
    \label{fig:quatre}
    \begin{minipage}{0.465\textwidth}
        \centering
        \includegraphics[width=\linewidth, height=5cm]{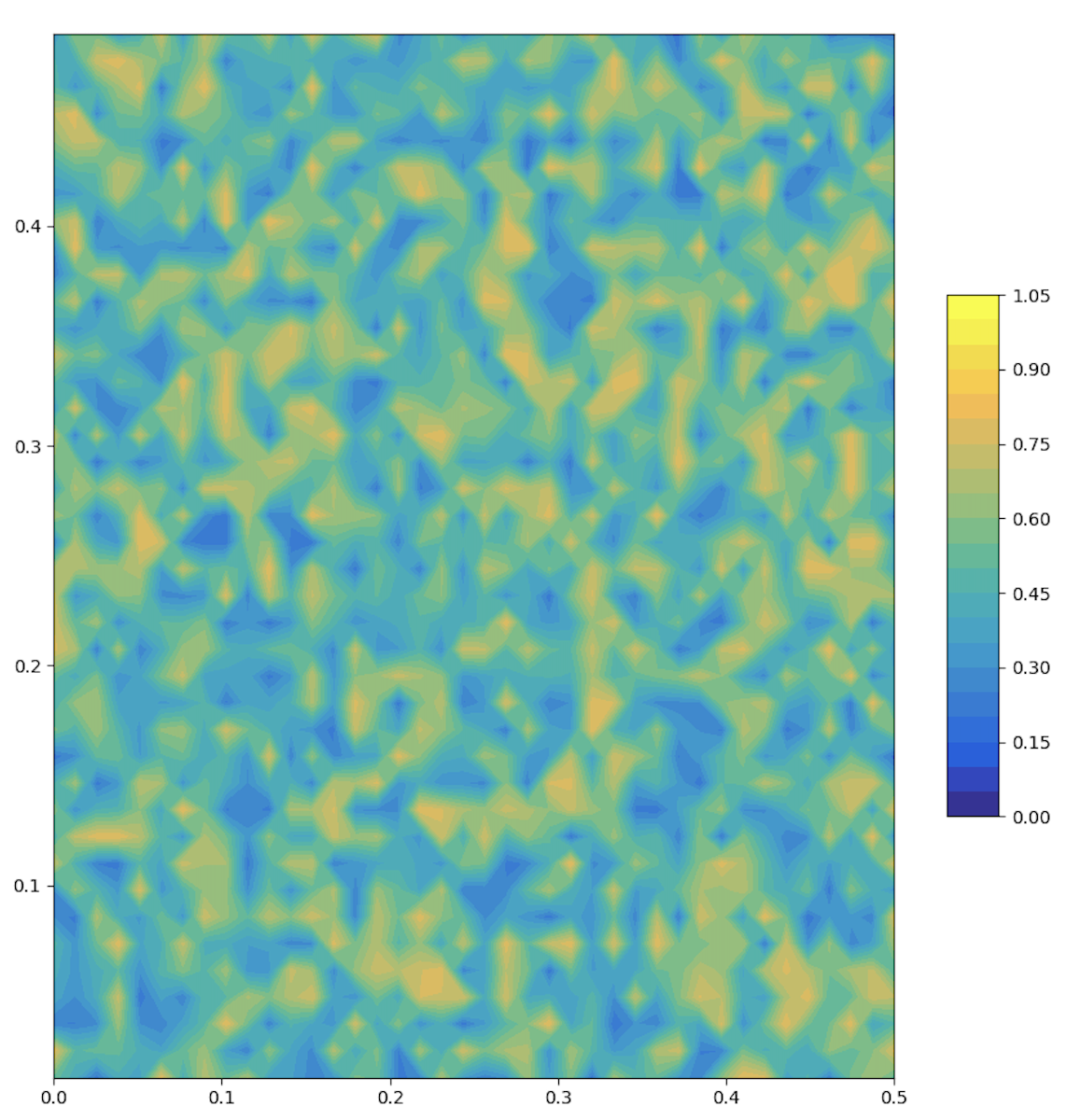}
        \caption*{Local PDE  -- Initial datum.}
        \label{fig:1}
    \end{minipage}
    \hfill
    \begin{minipage}{0.465\textwidth}
        \centering
        \includegraphics[width=\linewidth, height=5.2cm]{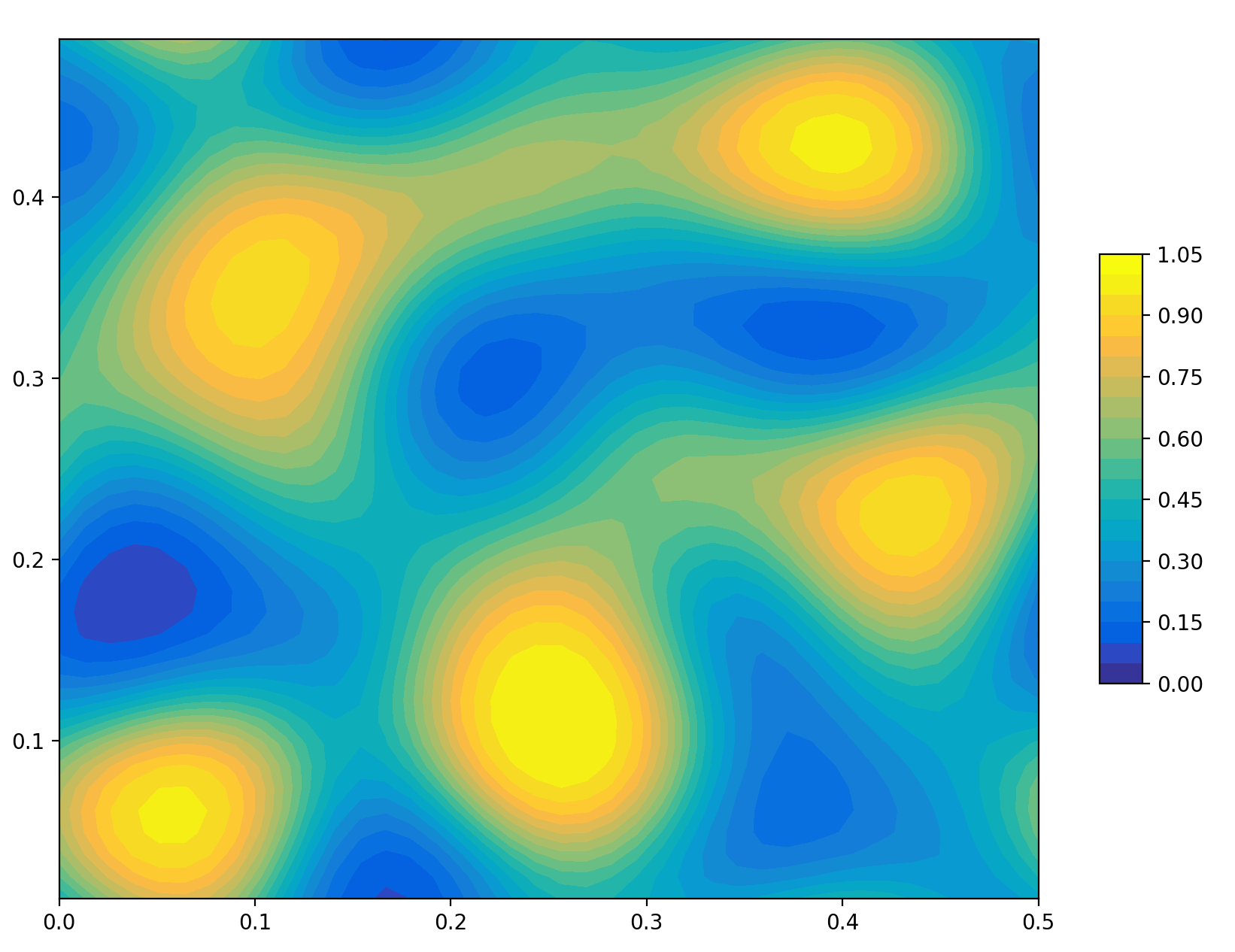}
        \caption*{Local PDE -- Evolution.}
        \label{fig:2}
    \end{minipage}
 \caption{Comparison of the evolution of the particle method for $\ee = 0.1$ and $N=750$ (top) with the evolution of the local PDE (bottom), leading to cluster formation in both. Case $m=2$. Let us note that for the particle method we neglect the term $\ee^\ast \dive(\rho\nabla R_{\alpha}\ast\rho)$. The initial conditions are random for both and are not the same.}\label{fig:Numerics}
\end{figure}

\bibliography{references}
\bibliographystyle{abbrv}
\end{document}